\documentclass[reqno, 10pt]{amsart}
\usepackage{extarrows}
\usepackage{graphicx}
\usepackage{amscd}
\usepackage{amsmath,extarrows,amsfonts,amssymb,amsthm,mathrsfs,xy,cite,color}
\xyoption{curve}
\usepackage{fancybox}
\newcommand{\rad}{\rm rad}
\newcommand{\Id}{\rm Id}

\numberwithin{equation}{section}
\begin{document}

\newtheorem{lemdef}{Lemma-Definition}[section]
\newtheorem{theorem}{Theorem}[section]
\newtheorem{corollary}[theorem]{Corollary}
\newtheorem{lemma}[theorem]{Lemma}
\newtheorem{proposition}[theorem]{Proposition}
\newtheorem{definition}[theorem]{Definition}
\newtheorem{rmk}[theorem]{Remark}
\newtheorem{examples}[theorem]{Examples}
\newtheorem{fact}[theorem]{Fact}

\title[Bi-Frobenius algebra structures]{Bi-Frobenius quantum complete intersections \\ with permutation antipodes}

\author[Hai Jin, Pu Zhang]{Hai Jin, \ \ Pu Zhang$^*$  \\ \\  School of Mathematical Sciences \\
Shanghai Jiao Tong University,  \ Shanghai 200240, \ China }
\thanks{This work was supported by the National Natural Science Foundation of China under Grant No. 12131015, and Natural Science Foundation of Shanghai under Grant No. 23ZR1435100.}
\thanks{pzhang$\symbol{64}$sjtu.edu.cn \ \ \ \  jinhaifyh$\symbol{64}$sjtu.edu.cn}
\thanks{$^*$ Corresponding author}
\begin{abstract} \ Quantum complete intersections $A= A({\bf q, a})$ are Frobenius algebras, but in the most cases they can not become Hopf algebras.
This paper aims to find bi-Frobenius algebra structures on $A$. A key step is the construction  of comultiplication,
such that $A$ becomes a  bi-Frobenius algebra. By introducing compatible permutation and permutation antipode,
a necessary and sufficient condition is found, such that $A$ admits a bi-Frobenius algebra structure
with permutation antipode; and if this is the case, then a concrete construction is explicitly given. Using this, intrinsic conditions only involving the structure coefficients $({\bf q, a})$ of $A$ are obtained, for
$A$ admitting a bi-Frobenius algebra structure with permutation antipode. When $A$ is a symmetric algebra, then $A$ admits a bi-Frobenius algebra structure with permutation antipode
if and only if there exists a compatible permutation $\pi$ with $A$ such that $\pi^2 = {\rm Id}$.
\end{abstract}
\maketitle

\section{\bf Introduction}

Frobenius algebras appear in various areas of mathematics, and play important roles in particular in representation theory, Hopf algebras, triangulated categories,
and topological quantum field theory (see e.g. \cite{Hap1988}, \cite{M1992}, \cite{K1999},  \cite{K2004}, \cite{SY2011}).
A bi-Frobenius algebra,  introduced by Y. Doi and M. Takeuchi \cite{DT2000}, roughly speaking, is a Frobenius algebra and a Frobenius coalgebra,
together with an antipode, satisfying some compatible conditions (Definition \ref{dt}). Finite-dimensional Hopf algebras  are bi-Frobenius algebras, but the converse is not true.
A bi-Frobenius algebra is a Hopf algebra if and only if its comultiplication is an algebra homomorphism (\cite{Haim2007}).
A Bi-Frobenius algebra shares many pleasant properties similar to Hopf algebras, and its basic theory is developed in \cite{DT2000}, \cite{Doi2002}, \cite{Doi2004}, \cite{Haim2007}, and \cite{Doi2010}.

\vskip5pt

A main problem in bi-Frobenius algebras is to find bi-Frobenius algebras which are not Hopf algebras.
The first such example is constructed in \cite{DT2000}; more examples
using group-like algebras or quiver techniques are given in \cite{Doi2004}, \cite{WZ2004}, \cite{Haim2007}, \cite{WC2007}, and \cite{WL2014}.

\vskip5pt

Let $n$ be an integer with $n\ge 2$,  \ ${\bf q} = (q_{ij})$ an \ $n\times n$ matrix over a field $\mathbb{K}$ with \ $q_{ii}=1$ and $q_{ij}q_{ji}=1$ for  \ $1\le i, j\le n$, and
\ ${\bf a}=(a_1, \cdots, a_n)$ an integral vector with \ all \ $a_i\ge  2$.  {\it A quantum complete intersection} $A= A({\bf q}, {\bf a})$ is the \ $\mathbb{K}$-algebra
\[
    \mathbb{K}\langle x_1, \cdots, x_n\rangle/\langle x_i^{a_i}, \ x_jx_i-q_{ij}x_ix_j, \ 1\le  i, \ j\le  n\rangle.
\]
Originated from quantum planes (\cite{Manin1987}), and got its present form in \cite{AGP1997}, a quantum complete intersection has been extensively studied in different contexts.
It is a quotient algebra of a quantum affine space (\cite{BG2002}) (or a skew polynomial ring in \cite{KKZ2010});
and closely related to braided Hopf algebras via quantum linear spaces (\cite{AH1998}).
Quantum complete intersections reveal many exotic homological and representation properties of algebras (\cite{LS1994}, \cite{ Ringel1996}, \cite{BO2008}, \cite{BE2011}, \cite{Mar},  \cite{RZ2020}); and their (co)homology groups
have been  studied (\cite{BGMS2005}, \cite{BE2008}, \cite{BO20082}, \cite{Oppermann2010}, \cite{EH2018}).

\vskip5pt

Quantum complete intersections $A= A({\bf q, a})$ are Frobenius algebras ([B]); however, in the most cases they can not become Hopf algebras.
In fact, $A$ admits a Hopf algebra structure if and only if it is commutative, the characteristic of the base field $\mathbb K$ is a prime $p$, and each component $a_i$ in vector ${\bf a}$ is a power of $p$ ([JZ, Theorem 4.1]).
In particular, if Char $\mathbb K = 0$ then $A$ can not become a Hopf algebra.
Thus, it is natural to study their possible bi-Frobenius structures. When a quantum complete intersection admits a bi-Frobenius algebra structure,
we simply call it {\it a  bi-Frobenius quantum complete intersection.}

\vskip5pt

This paper aims to study and construct bi-Frobenius quantum complete intersections.
If $A$ admits such a structure, then $A$ has some remarkable properties. For example, the canonical Nakayama automorphism
of $A$ is of order $2$: this is not true in general (see Proposition \ref{thenecessity} and Remark \ref{N2}).
A key step of constructing bi-Frobenius algebra structure on $A$ is the one of comultiplication,
such that $A$ becomes a  bi-Frobenius algebra. Since actions of antipodes of all the known bi-Frobenius algebra structures on $A$
are given by scalar multiplications, on the standard basis $\mathcal B$, up to a permutation $\pi$ in the symmetric group $S_n$, we study a permutation antipode (see Definition \ref{perS}).
Such a permutation $\pi$ is compatible with the structure coefficients $({\bf q, a})$ of $A$ (see Definition \ref{compatible}, Theorem \ref{thmperS}, and Proposition \ref{2perS}).
For quantum affine spaces, an analogue of compatible permutation and the automorphisms of permutation type have been already studied, e.g. in \cite{KKZ2010}, \cite {Ga2013},  \cite{CPWZ2016}, and \cite{BZ2017}. See also Remark \ref{rmka>2}.

\vskip5pt

The main result Theorem \ref{mainthm} gives a necessary and sufficient condition, such that $A$ admits a bi-Frobenius algebra structure
with permutation antipode. This condition involves
$$q_\pi = \prod\limits_{1\le j <k \le n} ({\bf q}^{\langle  \pi({\bf e}_k)|\pi({\bf e}_j) \rangle})^{(a_k-1)(a_j-1)}$$
where $\pi$ is a compatible permutation in $S_n$ (for details see Remark \ref{rmkcommutator} and Subsection 2.5). This element $q_\pi\in\mathbb K$ is in fact
$1$ or $-1$ (see Lemma \ref{lemmaI}(4)). Then  $A = A({\bf q,a })$ admits a bi-Frobenius algebra structure with permutation antipode
if and only if there exist a compatible permutation $\pi\in S_n$ with $\pi^2 = \Id$ and $c_i\in\mathbb K, \ 1\le i\le n$,  such that
$$c_ic_{\pi(i)}=h_{{\bf e}_i},  \ \ \ \ \ \ \ q_\pi \prod\limits_{1\le i\le n} c_i^{a_i-1} =1.$$
If this is the case, then a concrete construction is explicitly given. For details see Theorem \ref{mainthm} (and Subsection 2.5).
When $A$ is a symmetric algebra, then $A$ admits a bi-Frobenius algebra structure with permutation antipode
if and only if there exists a compatible permutation $\pi$ with $A$ such that $\pi^2 = \Id$ (Corollary \ref{symmetriccase}).

\vskip5pt

Based on Theorem \ref{mainthm}, intrinsic conditions only involving the structure coefficients $({\bf q, a})$ are obtained, for
$A$ admitting a bi-Frobenius algebra structure with permutation antipode, whatever $\sqrt{-1}$ is in $\mathbb K$ or not. See Theorems \ref{mainthm2} and \ref{mainthm3}.
The results obtained in this paper also provide a large class of bi-Frobenius algebras which are not Hopf algebras.

\vskip5pt

The paper is organized as follows. Preliminaries  are recalled in Section 2.
In Section 3 we study the structure coefficients ${\bf q}^{\langle {\bf \pi(u)|\pi(v)} \rangle}$ which will be used throughout. In Section 4,
the structural information, such as the right integrals, the right modular function, and the Nakayama automorphisms,
are explicitly given for those $A$ having a bi-Frobenius algebra structure. Section 5 studies the properties of bi-Frobenius quantum complete intersections with permutation antipode.
Section 6 gives the main result characterizing $A$ which admits a bi-Frobenius algebra structure with permutation antipode; and illustrates applications by some examples. In the final section,
intrinsic conditions only involving the structure coefficients of $A$ are obtained, for $A$ admitting a bi-Frobenius algebra structure with permutation antipode.

\section{\bf Preliminaries}

We consider a finite-dimensional algebra $A$ over a field $\mathbb{K}$.
The $\mathbb{K}$-linear dual  $A^* = \mathrm{Hom}_{\mathbb K}(A, \ \mathbb{K})$ of $A$ has a canonical $A$-$A$-bimodule structure: $(a\rightharpoonup f\leftharpoonup b)(x) = f(bxa),
\ \forall \ a, b, x \in A,  \ \forall \ f\in A^*$, where $\rightharpoonup$ and $\leftharpoonup$ denote the left  and the right action of $A$ on $A^*$, respectively.

\subsection{\bf Frobenius algebras and Nakayama automorphisms}

Recall that $A$ is {\it a Frobenius algebra} if $A\cong A^*$ as left $A$-modules,  or equivalently, as right $A$-modules.
If $A\cong A^*$ as  $A$-$A$-bimodules, then $A$ is called {\it a symmetric algebra}.
A more explicit (and equivalent) definition of a Frobenius algebra is a pair $(A, \phi)$, where $\phi\in A^*$ such that $A^* = A\rightharpoonup\phi$, or equivalently, $A^* = \phi\leftharpoonup A$.
If this is the case, then  $\phi$ is referred to {\it a Frobenius form} of $A$.
For more equivalent definitions and properties of a Frobenius algebra, we refer to K. Yamagata [Y], or A. Skowronski and K. Yamagata [SY].

\vskip5pt

Let $(A, \ \phi)$ be a Frobenius algebra. Then there is a unique algebra isomorphism $\mathcal{N}_\phi: A\longrightarrow A$
such that
\[
\phi(xy)=\phi(y\mathcal{N}_\phi(x)), \ \forall \ x, \ y\in A.
\]
It is called  {\it the Nakayama automorphism of $A$ with respect to $\phi$}. Let $U(A)$ denote the  set of invertible elements of $A$. One can find the following fact e.g. in [D3, Lemma 1.4(2)].

\begin{lemma} \label{nakayama} \ Let $(A, \phi)$ be a Frobenius algebra, $\phi'\in A^*$. Then $\phi'$ is also a Frobenius form of $A$
if and only if there is $z_1\in U(A)$ such that $\phi' = z_1\rightharpoonup\phi$; if and only if there is $z_2\in U(A)$ such that $\phi' = \phi\leftharpoonup z_2.$
If this is the case, then the Nakayama automorphism $\mathcal{N}_{\phi'}$ with respect to $\phi'$ is given by
$x\mapsto z_1\mathcal{N}_\phi(x)z_1^{-1}$ and also given by $x\mapsto \mathcal{N}_\phi(z_2xz_2^{-1})$.\end{lemma}

\subsection{\bf Frobenius coalgebras}
Assume that $(C, \Delta, \varepsilon)$ is a coalgebra over field $\mathbb{K}$. All tensor products are over $\mathbb{K}$. We use the Heyneman-Sweedler notation $\Delta(c)=\sum c_1\otimes c_2$. Then $C^*:=\operatorname{Hom}_{\mathbb{K}}(C,\ \mathbb{K})$ is an algebra; and $C$ becomes a left $C^*$-module: $f\rightharpoonup c = \sum c_1f(c_2), \ \forall \ f \in C^*, \ \forall \ c \in C$; and $C$ becomes a right $C^*$-module: $c \leftharpoonup f= \sum f(c_1)c_2$, where $\rightharpoonup$  and $\leftharpoonup$ denote the left and the right action of $C^*$ on $C$, respectively.

\vskip5pt

A finite-dimensional coalgebra $C$ is {\it a Frobenius coalgebra} if $C^*\cong C$ as left $C^*$-modules,  or equivalently, as right $C^*$-modules.
Also, $C$ is a Frobenius coalgebra if and only if  there is an element
$t\in C$ such that $C=t\leftharpoonup C^*$, or equivalently, $C = C^*\rightharpoonup t$. If this is the case, the Frobenius coalgebra $C$ will be denoted by the pair $(C, t)$.

\subsection{\bf Bi-Frobenius algebras}

\begin{definition} {\rm (\cite{DT2000})} \label{definitionbf} \ Let $A$ be a finite-dimensional algebra and a coalgebra over a field $\mathbb{K}$, \ $t\in A$, and $\phi\in A^*$.
Let $S: A\longrightarrow A$ be the $\mathbb K$-linear map defined by $S(a)= \sum\phi(t_{1}a)t_{2}, \ \forall \ a\in A.$
The quadruple   $(A, \ \phi, \ t, \ S)$, or simply, $A$ is  a bi-Frobenius algebra, if the conditions are satisfied$:$

\vskip5pt

{\rm (i)} \ The counit $\varepsilon: A\longrightarrow \mathbb{K}$ is an algebra homomorphism $($thus $\varepsilon(1_A) = 1_A);$ and the identity $1_A$ is a group-like element
$($i.e.,  $\Delta(1_A) = 1_A\otimes 1_A);$

\vskip5pt

    {\rm (ii)} \ $(A, \ \phi)$ is a Frobenius algebra$;$ and \ $(A, \ t)$ is a Frobenius coalgebra$;$

    \vskip5pt

{\rm (iii)} \  $S$ is an algebra anti-homomorphism of $A$ $($i.e., $S(1) = 1$ and $S(ab) = S(b)S(a), \ \forall \ a, b\in A)$,
and $S$ is a coalgebra anti-homomorphism of  $A$ $($i.e., $\varepsilon \circ S =\varepsilon$, and \  $\Delta(S(a)) = \sum S(a_{2})\otimes S(a_{1}), \ \forall \ a\in A)$.

 \vskip5pt

    In this case, $S$ is called the antipode of bi-Frobenius algebra  $(A, \ \phi, \ t, \ S)$.

\end{definition}

The antipode $S$ of a bi-Frobenius algebra  $(A, \phi, t, S)$ is bijective, since $S = \kappa\circ \theta$, where $\theta: A\longrightarrow A^*, \ 1\mapsto \phi$ is the isomorphism of left $A$-modules,
and $\kappa: A^* \longrightarrow A, \ f\mapsto t\leftharpoonup f$ (or equivalently, $\varepsilon\mapsto t$) is the isomorphism of right $A^*$-modules. Usually  $S^{-1}$ is called
{\it the composite inverse} of $S$, and denoted by $\overline{S}$.

\vskip5pt

Let  $(A, \phi, t, S)$ be a bi-Frobenius algebra. By identifying $A$ with $A^{**}$, then $(A^*, \ t, \ \phi, \ S^*)$ is a bi-Frobenius algebra. Also,
$(A^{\rm op}, \ \phi,  \ \overline{S}(t), \ \overline{S})$ and $(A^{\rm cop}, \ \phi\circ \overline{S},  \ t, \ \overline{S})$ are bi-Frobenius algebras,
where $A^{\rm op}$ means the multiplication is the opposite and the comultiplication is the same,  $A^{\rm cop}$ means the multiplication is the same and the comultiplication is the opposite.

\subsection{\bf The Nakayama automorphism and the $S^4$-formula for a Bi-Frobenius algebra}

A $\mathbb{K}$-algebra $A$ is {\it augmented}, if there is an algebra homomorphism \ $\varepsilon: A\longrightarrow \mathbb{K}$, which is called {\it an augmentation}.
For an augmented algebra \ $(A, \varepsilon)$, an element \ $t\in A$ is said to be {\it a right} (respectively, {\it left}) {\it integral} of \ $A$ if \ $t x=t \varepsilon(x)$ (respectively, \ $x t= \varepsilon(x)t$), for all \ $x\in A$. If the right integral space coincides with the left integral space, then \  $A$ is called {\it unimodular}.

\vskip5pt

Let $(A, \phi)$ be an augmented Frobenius algebra with augmentation $\varepsilon: A\longrightarrow \mathbb{K}.$ Then $\varepsilon = t\rightharpoonup\phi = \phi \leftharpoonup s$ for some unique $t, s\in A$;
it can be easily deduced that $t$ and $s$ are respectively left and right integrals of $A$.

\vskip5pt

For a finite-dimensional coalgebra $C$, the algebra $C^*$ (its identity is $\varepsilon$) is augmented if and only if $C$ has a group-like element \ $h$.
If this is the case,  then a right (respectively, left) integral of algebra $C^*$ can be interpreted as an element $\phi\in C^*$, such that
\ $x \leftharpoonup\phi =\phi(x)h$  (respectively, \ $\phi \rightharpoonup x= h\phi(x)$), for all \ $x\in C$.

Let $(C, \ t)$ be a Frobenius coalgebra with group-like element $h.$ Then $h = \phi\rightharpoonup t = t \leftharpoonup \psi $ for some unique $\phi, \psi\in C^*$; it can be easily deduced that $\phi$ and $\psi$ are respectively left and right integrals of $C^*$.

\begin{lemma} \label{dt} $($\cite{DT2000}, \ \cite {Haim2007}$)$  \ $(1)$ \ If \ $A$ is an augmented Frobenius algebra, then the space of right $($respectively, left$)$ integrals is one-dimensional.

\vskip5pt

$(2)$ \  Let \ $(A, \ \phi, \ t, \ S)$ be a bi-Frobenius algebra. Then

{\rm (i)} \ $\varepsilon = \phi\leftharpoonup t;$

{\rm (ii)} \ $\phi(t) = 1;$

{\rm (iii)} \ $t$ is a right integral of $A;$ and the space of right integrals of $A$ is $\mathbb Kt;$

{\rm (iv)} \  $1_A = t\leftharpoonup \phi;$

{\rm (v)} \ $\phi$ is a right integral of $A^*;$ and the space of right integrals of $A^*$ is $\mathbb K\phi$.

{\rm (vi)} \ If $A$ is unimodular, then $S(t)=t$.

\vskip5pt

$(3)$ \  If \ $(A,\ \phi,\ t,  S)$ is a bi-Frobenius algebra, then $(A,  \ c\phi, \ \frac{1}{c}t,  S)$ is also a bi-Frobenius algebra  for \ $0\ne c\in \mathbb K$.
\end{lemma}
\begin{proof} \ The fact that $t$ is a right integral of $A$ mainly use $\varepsilon \circ S = \varepsilon.$ By the definition of $S$ one sees that
$t\leftharpoonup \phi = S(1_A) = 1_A$; and from this one gets that $\phi$ is a right integral of $A^*$.

\vskip5pt

We include a proof of (vi) in (2). By definition $S(t)=\sum \phi(t_1t)t_2.$
Since $A$ is unimodular, $t$ is also a left integral, and hence $xt = \varepsilon(x)t, \ \forall \ x\in A$. Thus
$$S(t)=\sum \phi(t_1t)t_2 = \sum \phi(\varepsilon(t_1)t)t_2 =\phi(t) \sum \varepsilon(t_1)t_2 = \phi(t)t.
$$
By (ii) in (2) one has $\phi(t)=1$, and hence $S(t)=t$. \end{proof}

\vskip5pt

Let  $(A, \ \phi, \ t, \ S)$ be a bi-Frobenius algebra. By Lemma \ref{dt}(2) one has \ $\varepsilon = \phi\leftharpoonup t$ and $1_A = t\leftharpoonup \phi$. Put
$$\alpha: = t\rightharpoonup\phi\in A^*,  \ \ \ \ {\mathfrak{a}}: = \phi\rightharpoonup t\in A.$$
Following [DT], $\alpha$  is called the {\it right modular function}, and ${\mathfrak{a}}$ is called the {\it right modular element}.

\begin{lemma} {\rm (\cite{DT2000}, 3.3; 3.4; 3.6)} \label{lemmanak} \ Let $(A, \ \phi, \ t, \ S)$ be a bi-Frobenius algebra. Then

\vskip5pt

$(1)$ \  $\alpha$ is invertible in $A^*$ and $\alpha^n\in A^*$ is an algebra homomorphism from $A$ to $k$ for $n\in\mathbb{Z};$
and $\mathfrak{a}$ is invertible in $A$ and \ $\mathfrak{a}^n$ is a group-like element of $A$ for $n\in\mathbb{Z}.$
    \vskip5pt
    $(2)$  \ $S(\mathfrak{a})=\mathfrak{a}^{-1}.$
    \vskip5pt
    $(3)$ \ $\mathcal{N}_\phi(x) = \mathfrak{a}^{-1}S^2(\alpha\rightharpoonup x)\mathfrak{a}, \ \forall \ x\in A,$  where  $\mathcal N_\phi$ is the Nakayama automorphism of $A$ with respect to $\phi$.
    \vskip5pt
    $(4)$ \ $S^4(x)=\mathfrak{a}(\alpha^{-1}\rightharpoonup x\leftharpoonup\alpha)\mathfrak{a}^{-1}=\alpha^{-1}\rightharpoonup(\mathfrak{a} x\mathfrak{a}^{-1})\leftharpoonup\alpha,  \ \forall \ x\in A$.
\end{lemma}

Lemma \ref{lemmanak}(4) generalizes Radford's $S^4$-formula for finite-dimensional Hopf algebras (\cite{Radford1976}).

\subsection{\bf Quantum complete intersections.} \ Denote by $\mathbb{N}_0$  the set of non-negative integers. Let \ ${\bf a}=(a_1, \cdots, a_n)\in \mathbb{N}_0^n$ with \ $n\ge 2$ and all \ $a_i\ge  2$, and \ ${\bf q} = (q_{ij})$ an \ $n\times n$ matrix
over  $\mathbb{K}$ with \ $q_{ii}=1$ and $q_{ij}q_{ji}=1$ for  \ $1\le i, j\le n$. {\it A quantum complete intersection} $A= A({\bf q}, {\bf a})$ is the \ $\mathbb{K}$-algebra
\[
    \mathbb{K}\langle x_1, \cdots, x_n\rangle/\langle x_i^{a_i}, \ x_jx_i-q_{ij}x_ix_j, \ 1\le  i, \ j\le  n\rangle.
\]
If all $a_i = 2$, then it is called {\it a quantum exterior algebra}.

\vskip5pt

Put
\[
V=\{ \mathbf{v} = (v_1, \cdots, v_n)\in \mathbb{N}_0^n \ | \ v_i\le a_i-1,  \ 1\le i\le n\}.
\]
For $\mathbf{v}\in \mathbb N_0^n$, write $x_{\mathbf{v}} = x_1^{v_1}\cdots x_n^{v_n}\in A$. Then $\mathcal B = \{x_{\mathbf{v}} \ | \ \mathbf{v}\in V\}$ is a basis of $A$, which will be called {\it the standard basis}.
The dual basis of \ $A^*$  is \ $\{x_{\bf v}^* \ |  \ {\bf v}\in V\}$, where
\ $x_{\bf v}^*(x_{\bf u}) = \delta_{{\bf v}, {\bf u}}, \ \forall \ {\bf v}, {\bf u}\in V$, $\delta_{{\bf v}, {\bf u}}$ is the Kronecker symbol. Note that $A$ is a local algebra, with the unique maximal ideal $\rad A = \langle x_1, \cdots,x_n \rangle$.

\vskip5pt

{\it The degree} of $x_{\mathbf{v}}\in \mathcal B$ is defined as \ $|x_{\mathbf{v}}| = \sum\limits_{i=1}^n v_i$,
and the degree of $x_{\mathbf{u}} \otimes x_{\mathbf{v}}\in \mathcal B\otimes \mathcal B$ is defined as \ $|x_{\mathbf{u}}|+|x_{\mathbf{v}}|$.
Then $A$ is an $\mathbb N_0$-graded algebra with respect to the degree. A $\mathbb K$-linear map $S: A\longrightarrow A$ is said to {\it graded},
provided that $S$ preserves the degrees of all elements $x_{\mathbf{v}}\in \mathcal B$.

\vskip5pt

Put \ ${\bf 0} = (0, \cdots, 0), \ \ {\bf 1} = (1, \cdots, 1)$. For $1\le i\le n$, let \ ${\bf e}_i\in\mathbb{N}_0^n$ \ be the vector with the \ $i$-th component \ $1$ and other components \ $0$.
Define a partial order on $\mathbb{N}_0^n$  by \ $\mathbf{u}\le  \mathbf{v}$ if and only if \ $u_i\le  v_i$ for all $1\le i\le n$.
Note that $x_{\bf v} \ne 0$ if and only if  ${\bf v}\le \mathbf{a-1}$. Thus, if  $|x_{\mathbf{v}}| > (\sum\limits_{i = 1}^na_i) - n$ then
$x_{\bf v} = 0$.

\vskip5pt

For $\mathbf{u, v} \in \mathbb{Z}^n$, following S. Oppermann \cite[Section 2]{Oppermann2010}, put
$${\bf q}^{\langle \mathbf{u}|\mathbf{v}\rangle}= \prod\limits_{1\le i<j\le n}q_{ij}^{u_jv_i}; \ \ \  \ \ h_{\bf v}:=\frac{{\bf q}^{\bf \langle a-1-v|v\rangle}}{{\bf q}^{\bf \langle v|a-1-v\rangle}}.$$
In particular, \[{\bf q}^{\langle \mathbf{0}|\mathbf{v}\rangle}= 1 = {\bf q}^{\langle \mathbf{u}|\mathbf{0}\rangle}; \ \ \ h_{\bf 0} = 1 = h_{\bf a-1}; \ \ \
 h_{{\bf e}_i} =\prod\limits_{1\le j\le n} q_{i j}^{a_j-1}, \ \ 1\le i\le n.
\]
For convenience, we will call ${\bf q}^{\langle \mathbf{u}|\mathbf{v}\rangle}, \ \forall \ \mathbf{u, v} \in \mathbb{Z}^n,$
{\it the structure coefficients} of $A$.

\begin{lemma} {\rm (S. Oppermann)} \label{propeqq} \  The following equalities hold for any ${\bf u,v,w}\in \ \mathbb{Z}^n:$
\vskip5pt
${\rm(1)}$ \  \ $x_\mathbf{u}x_\mathbf{v}={\bf q}^{\langle \mathbf{u}|\mathbf{v}\rangle}x_{\mathbf{u+v}}, \ \forall \ {\bf u, v}\in \mathbb N_0^n.$

\vskip5pt
${\rm(2)}$ \  \ ${\bf q}^{\langle  {\bf e}_j|{\bf e}_k\rangle}={\bf q}^{\langle  {\bf e}_k|{\bf e}_j\rangle}q_{kj}; \ \ \ {\bf q}^{\langle  {\bf e}_j|{\bf e}_k\rangle} = \left\{ \begin{array}l 1, \ \ \ j\le k;\\ q_{kj}, \ j>k.   \end{array} \right. $
\vskip5pt
    ${\rm(3)}$ \  \ ${\bf q}^{\langle {\bf -u|v} \rangle}= {\bf q}^{\langle {\bf u|-v} \rangle} = \frac{1}{{\bf q}^{\langle {\bf u|v}\rangle}};$ \ \ \  ${\bf q}^{\langle {\bf u+v|w} \rangle}={\bf q}^{\langle {\bf u|w} \rangle}{\bf q}^{\langle {\bf v|w}\rangle};$  \ \ \ ${\bf q}^{\langle {\bf u|v+w} \rangle}={\bf q}^{\langle {\bf u|v} \rangle}{\bf q}^{\langle {\bf u|w}\rangle}.$
    \vskip5pt
    ${\rm(4)}$ \ \ ${\bf q}^{\langle{\bf u|v}\rangle}=\prod\limits_{1\le i,j\le n}({\bf q}^{\langle{{\bf e}_i|{\bf e}_j}\rangle})^{u_iv_j}
    =\prod\limits_{1\le j< i\le n}({\bf q}^{\langle{{\bf e}_i|{\bf e}_j}\rangle})^{u_iv_j}.$
    \vskip5pt
    ${\rm(5)}$ \  \ $h_{\bf u+v} = h_{\bf u}h_{\bf v}; \ \ \ h_{\bf u} = \prod\limits_{1\le i\le n}h_{{\bf e}_i}^{u_i}.$
\end{lemma}
\begin{proof} All these equalities can be directly verified. For example, by definition one has
\vskip5pt
${\bf q}^{\langle {\bf u+v|w} \rangle} = \prod\limits_{1\le i<j\le n}q_{ij}^{(u_j+v_j)w_i}= \prod\limits_{1\le i<j\le n}q_{ij}^{u_jw_i}\prod\limits_{1\le i<j\le n}q_{ij}^{v_jw_i}={\bf q}^{\langle {\bf u|w}\rangle}{\bf q}^{\langle{\bf v|w} \rangle}.$
\end{proof}

\vskip5pt

It is known (see e.g. the proof of [B, Lemma 3.1]) that  $x_\mathbf{a-1}^*$ is a Frobenius form of $A$.
Thus, any Frobenius form of $A$ is of the form \ $\phi= x_\mathbf{a-1}^*\leftharpoonup  z$ with \ $z\in U(A)$ (cf. Lemma \ref{nakayama}).
 Since $A$ is finite-dimensional and local,
$z = \sum\limits_{\mathbf{v}\in V} c'_{\mathbf{v}} x_\mathbf{v}\in A$ with each  \ $c'_{\mathbf{v}} \in \mathbb K$ and $c'_{\mathbf{0}} \ne 0$.
Using $x_\mathbf{a-1}^*\leftharpoonup x_\mathbf{v} = {\bf q}^{\langle \mathbf v | \mathbf {a-1-v} \rangle}x^*_{\mathbf{a-1-v}}$ one has
$$\phi = x_{\mathbf{a-1}}^*\leftharpoonup z = x_\mathbf{a-1}^* \leftharpoonup \sum\limits_{\mathbf{v}\in V} c'_{\mathbf{v}} x_\mathbf{v}
= \sum\limits_{\mathbf{v}\in V} c'_\mathbf{v}{\bf q}^{\langle \mathbf v | \mathbf {a-1-v} \rangle} \ x^*_{\mathbf{a-1-v}}.$$
Write \ $\phi = \sum\limits_{\mathbf{v}\in V} c'_\mathbf{v}{\bf q}^{\langle \mathbf v | \mathbf {a-1-v} \rangle} \ x^*_{\mathbf{a-1-v}}$ \ as
$\sum\limits_{\mathbf{u}\in V} c_\mathbf{u}x_\mathbf{u}^*$. Then  $c_{\mathbf{a-1}} = c'_{\mathbf{0}}\neq 0$.

\begin{lemma} \label{frobeniushomomorphism} $(${\rm P. A. Bergh [B]}$)$ \ Quantum complete intersection $A = A({\bf q, a})$ is a local Frobenius algebra.
It is symmetric if and only if \ $h_{{\bf e}_i} =\prod\limits_{1\le j\le n} (q_{i j})^{a_j-1} = 1, \ \ 1\le i\le n.$

\vskip5pt

Any Frobenius form of $A$ is of the form  $\phi = x^*_{\bf a -1} \leftharpoonup z$ with $z\in U(A);$
and \ $\phi = \sum\limits_{\mathbf{v}\in V} c_\mathbf{v}x_\mathbf{v}^* \ \ \mbox{with} \ c_\mathbf{v} \in\mathbb K \  \mbox{and} \  c_{\mathbf{a-1}} \neq 0$
gives all  the Frobenius homomorphisms of $A$.

\vskip5pt

The Nakayama automorphism $\mathcal N_{x_{\bf a-1}^*}$ of $A$ with respect  to the Frobenius form $x_{\bf a-1}^*$ is given by
\[
  \mathcal{N}_{x_{\bf a-1}^*}(x_{\bf v})= \frac{{\bf q}^{\bf \langle a-1-v|v\rangle}}{{\bf q}^{\bf \langle v|a-1-v\rangle}}x_{\bf v} = h_{\bf v} x_{\bf v}, \ \forall \ {\bf v}\in V.
\]
\end{lemma}

We will call $\mathcal N_{x_{\bf a-1}^*}$ {\it the canonical Nakayama automorphism} of $A$, and denote it simply by $\mathcal N$. Thus
$\mathcal N(x_{\bf v}) = h_{\bf v} x_{\bf v}, \ \forall \ {\bf v}\in V.$

\section{\bf Acting with $S_n$}
In order to study the bi-Frobenius algebra structure of a quantum complete intersection $A = A({\bf q, a})$ with permutation antipode
(see Definition \ref{perS}), we need to study the structure coefficients ${\bf q}^{\langle {\bf \pi(u)|\pi(v)} \rangle}$ for each permutation $\pi\in S_n$.

\vskip5pt

The equations we present here will be referenced later and are of independent interest. The reader may choose to skip this section initially and refer to it later when necessary.

\subsection{Structure coefficients under permutations}  Let $\pi\in S_n$. Then $\pi$ induces an action on $\mathbb Z^n$, still denoted by $\pi$,  by
\begin{align}\pi({\bf v}) =(v_{\pi^{-1}(1)}, \cdots, v_{\pi^{-1}(n)}), \ \ \forall \ {\bf v} = (v_1, \cdots, v_n)\in \mathbb Z^n.\end{align}
We will use this induced action throughout, without further explanations.

\begin{lemma}\label{commutator} Let $\pi\in S_n$ and ${\bf u,v}\in \mathbb{Z}^n$. Then

\vskip5pt
    {\rm (1)} \ \  $\pi({\bf 0}) = {\bf 0}; \ \ \pi({\bf 1}) = {\bf 1}; \ \ \pi({\bf e}_i) = {\bf e}_{\pi(i)}, \ 1\le i\le n; \ \ \pi({\bf u+v})= \pi({\bf u})+\pi({\bf v}).$

\vskip5pt
    {\rm (2)} \ \ For \ ${\bf v}\in \mathbb{N}_0^n$, there holds the equality in quantum complete intersection $A= A(\bf{q, a}):$
\begin{align*} x_{\pi ({\bf e}_n)}^{v_n}x_{\pi ({\bf e}_{n-1})}^{v_{n-1}}\cdots x_{\pi ({\bf e}_1)}^{v_1}
&= \prod_{1\le j< k\le n}  ({\bf q}^{\langle \pi({{\bf e}_k}) |\pi({{\bf e}_j})\rangle})^{v_kv_j} x_{\pi({\bf v})} \\ & = \prod_{1\le j< k\le n}  ({\bf q}^{\langle \pi({{\bf e}_k}) |\pi({{\bf e}_j})\rangle})^{v_kv_j}
x_1^{v_{\pi^{-1}(1)}}\cdots x_n^{v_{\pi^{-1}(n)}}.\end{align*}

\vskip5pt
    {\rm (3)} \ \ ${\bf q}^{\langle {\bf \pi(u)|\pi(v)} \rangle}
    = \prod\limits_{1\le j,k\le n} ({\bf q}^{\langle  \pi({\bf e}_j)|\pi({\bf e}_k) \rangle})^{u_jv_k}.$
\end{lemma}
\noindent {\it Proof.}  \ The assertion (1) follows directly from (3.1).

\vskip5pt

(2) \ Using   $x_{\bf \pi (u)}x_{\bf\pi (v)}= {\bf q}^{\langle {\bf\pi(u) |\pi(v)} \rangle}x_{\bf\pi(u+v)}$ (cf. Lemma \ref{propeqq}(1)), one gets
$$x_{\pi ({\bf e}_n)}^{v_n}\cdots x_{\pi ({\bf e}_1)}^{v_1} = ({\bf q}^{\langle \pi({{\bf e}_n}) |\pi({{\bf e}_{n-1}}) \rangle})^{v_nv_{n-1}} x_{\pi(v_n{{\bf e}_n}+v_{n-1}{{\bf e}_{n-1}})}
x_{\pi({\bf e}_{n-2})}^{v_{n-2}}\cdots x_{\pi({\bf e}_1)}^{v_1}.$$
Again by  $x_{\bf \pi (u)}x_{\bf\pi (v)}= {\bf q}^{\langle {\bf\pi(u) |\pi(v)} \rangle}x_{\bf\pi(u+v)}$ one has
\begin{align*}
&  x_{\pi(v_n{{\bf e}_n}+v_{n-1}{{\bf e}_{n-1}})}x_{\pi({\bf e}_{n-2})}^{v_{n-2}}\cdots x_{\pi({\bf e}_1)}^{v_1}
  \\ = &
  {\bf q}^{\langle \pi(v_n{{\bf e}_n}+v_{n-1}{{\bf e}_{n-1}}) |\pi(v_{n-2}{{\bf e}_{n-2}})\rangle}x_{\pi(v_n{{\bf e}_n}+v_{n-1}{{\bf e}_{n-1}}+v_{n-2}{{\bf e}_{n-2}})}x_{\pi({\bf e}_{n-3})}^{v_{n-3}}\cdots x_{\pi({\bf e}_1)}^{v_1}\end{align*}
Continuing this process, and using ${\bf q}^{\langle {\bf u+v |w}\rangle } = {\bf q}^{\langle \bf {u|w}\rangle }{\bf q}^{\langle \bf v|w\rangle}$ (cf. Lemma \ref{propeqq}(3)),  one gets
\begin{align*} x_{\pi ({\bf e}_n)}^{v_n}x_{\pi ({\bf e}_{n-1})}^{v_{n-1}}\cdots x_{\pi ({\bf e}_1)}^{v_1} & = \prod_{1\le j< k\le n}  ({\bf q}^{\langle \pi({{\bf e}_k}) |\pi({{\bf e}_j})\rangle})^{v_kv_j}
x_{\pi(v_n{{\bf e}_n}+\cdots +v_1{{\bf e}_1})}
\\ & = \prod_{1\le j< k\le n}  ({\bf q}^{\langle \pi({{\bf e}_k}) |\pi({{\bf e}_j})\rangle})^{v_kv_j} x_{\pi({\bf v})}.\end{align*}
\vskip5pt
    (3) \  By  ${\bf q}^{\langle {\bf u|v} \rangle} =\prod\limits_{1\le j,k\le n} ({\bf q}^{\langle {\bf e}_j|{\bf e}_k \rangle})^{u_jv_k}$ (cf. Lemma \ref{propeqq}(4)) one has
\begin{align*}{\bf q}^{\langle {\bf \pi(u)|\pi(v)} \rangle}
    & = \prod\limits_{1\le j,k\le n} ({\bf q}^{\langle  {\bf e}_j|{\bf e}_k \rangle})^{u_{\pi^{-1}(j)}v_{\pi^{-1}(k)}}
    \\ & = \prod\limits_{1\le s,t\le n} ({\bf q}^{\langle  {\bf e}_{\pi(s)}|{\bf e}_{\pi(t)}\rangle})^{u_sv_t}
   \\ & \xlongequal{{\bf e}_{\pi(s)}=\pi({\bf e}_s)} \prod\limits_{1\le s,t\le n} ({\bf q}^{\langle  \pi({\bf e}_s)|\pi({\bf e}_t) \rangle})^{u_sv_t}, \ \forall\ {\bf u,v}\in V. \ \ \ \ \ \ \ \ \ \ \ \ \ \ \ \ \ \ \ \ \square \end{align*}

\begin{rmk}\label{rmkcommutator} Taking \ ${\bf v} = {\bf a-1}$ in {\rm Lemma \ref{commutator}(2)} one gets the equality in $A= A(\bf{q, a}):$
$$x_{\pi ({\bf e}_n)}^{a_n-1}x_{\pi ({\bf e}_{n-1})}^{a_{n-1}-1}\cdots x_{\pi ({\bf e}_1)}^{a_1-1} = \prod_{1\le j< k\le n}  ({\bf q}^{\langle \pi({{\bf e}_k}) |\pi({{\bf e}_j})\rangle})^{(a_k-1)(a_j-1)} x_{\pi({\bf a-1})}.$$
This  will be used later. For convenience, denote by
$$q_\pi = \prod_{1\le j< k\le n}  ({\bf q}^{\langle \pi({{\bf e}_k}) |\pi({{\bf e}_j})\rangle})^{(a_k-1)(a_j-1)}.$$
It will play an important role in studying bi-Frobenius algebra structures on $A$. \end{rmk}

\subsection{Structure coefficients under compatible permutations} In order to study bi-Frobenius algebra structure  with permutation antipode, we need the following

\vskip5pt

\begin{definition} \label{compatible} A permutation $\pi\in S_n$ is called a compatible permutation with $A(\bf q, a)$, or simply, {\it a compatible permutation},
provided that
$$a_{\pi(i)}=a_{i}, \ \ \ \ q_{\pi(i)\pi(j)}=q_{ji}, \  \ \  1\le i, j\le n.$$
\end{definition}

Compatible permutations exist in many situations. Note that when $\bf q$ is a symmetric matrix then  $\pi = \Id$ is a compatible permutation.
A similar notion of compatible permutation has been already used, e.g. in \cite{KKZ2010}, \cite {Ga2013},  \cite{CPWZ2016}, and \cite{BZ2017},
in studying isomorphisms of quantum affine spaces.

\begin{lemma}\label{propusefuleq2} For  $A = A({\bf q, a})$, let $\pi$ be a compatible permutation in $S_n$.
Then the following equalities hold for  ${\bf u,v}\in \mathbb Z^n:$
    \vskip5pt
    {\rm (1)} \ \
    $\pi({\bf a-1})={\bf a -1}.$
    \vskip5pt
    {\rm (2)} \ \ ${\bf q}^{\langle {\bf u|v} \rangle}{\bf q}^{\langle {\bf \pi(u)|\pi(v)} \rangle}={\bf q}^{\langle {\bf v|u} \rangle}{\bf q}^{\langle {\bf \pi(v)|\pi(u)} \rangle}.$
    \vskip5pt
    {\rm (3)} \ \ ${\bf q}^{\langle {\bf \pi(v)|\pi(u)}\rangle} ={\bf q}^{\langle{\bf u|v}\rangle} \prod\limits_{1\le j<k \le n} ({\bf q}^{\langle  \pi({\bf e}_k)|\pi({\bf e}_j) \rangle})^{u_jv_k+u_kv_j}.$

    \vskip5pt
    {\rm (4)} \ \ $h_{\bf v}h_{\bf \pi(v)}=1$.
\end{lemma}
\begin{proof} (1) \  Since  $a_{\pi(i)} =a_{i}, \ 1\le i\le n$, thus  $a_{\pi^{-1}(i)}=a_{i}, \ 1\le i\le n$, and hence
\ $\pi({\bf a - 1}) = (a_{\pi^{-1}(1)} - 1, \cdots,\ a_{\pi^{-1}(n)}-1)  = (a_{1}-1, \cdots, \ a_{n}-1)= {\bf a -1}.$

\vskip5pt

    (2) \ By Lemma \ref{propeqq}(2) and the assumption $q_{\pi(j)\pi(k)}=q_{kj}, \ 1\le j,k\le n$, and note that $\pi({\bf e}_j) = {\bf e}_{\pi(j)}$, one has
\ $\frac{{\bf q}^{\langle {\bf e}_j |{\bf e}_k \rangle}}{{\bf q}^{\langle {\bf e}_k |{\bf e}_j \rangle}}=q_{kj} =q_{\pi(j)\pi(k)} = \frac{{\bf q}^{\langle \pi({\bf e}_k) |\pi({\bf e}_j) \rangle}}{{\bf q}^{\langle \pi({\bf e}_j) |\pi({\bf e}_k) \rangle}}, \ \ 1\le j,k\le n.$ Then by Lemma \ref{commutator}(3) one has
    \begin{align*}
    \frac{{\bf q}^{\langle {\bf \pi(u)|\pi(v)} \rangle}}{{\bf q}^{\langle {\bf \pi(v)|\pi(u)} \rangle}}
    & = \frac{\prod\limits_{1\le j,k\le n} ({\bf q}^{\langle  \pi({\bf e}_j)|\pi({\bf e}_k) \rangle})^{u_jv_k}}{\prod\limits_{1\le j,k\le n} ({\bf q}^{\langle  \pi({\bf e}_j)|\pi({\bf e}_k) \rangle})^{v_ju_k}} = \frac{\prod\limits_{1\le j,k\le n} ({\bf q}^{\langle  \pi({\bf e}_k)|\pi({\bf e}_j) \rangle})^{u_kv_j}}{\prod\limits_{1\le j,k\le n} ({\bf q}^{\langle  \pi({\bf e}_j)|\pi({\bf e}_k) \rangle})^{v_ju_k}}  \\
    &
    = \prod\limits_{1\le j,k\le n} (\frac{{\bf q}^{\langle  \pi({\bf e}_k)|\pi({\bf e}_j) \rangle}}{{\bf q}^{\langle  \pi({\bf e}_j)|\pi({\bf e}_k) \rangle}})^{u_kv_j}
    = \prod\limits_{1\le j,k\le n} (\frac{{\bf q}^{\langle {\bf e}_j |{\bf e}_k \rangle}}{{\bf q}^{\langle {\bf e}_k |{\bf e}_j \rangle}})^{u_kv_j}\\
    &\xlongequal{{\rm By \ Lem.\ref{propeqq} (4)}}  \frac{{\bf q}^{\langle {\bf v|u} \rangle}}{{\bf q}^{\langle {\bf u|v} \rangle}}.
    \end{align*}

    \vskip5pt

    (3) \  By Lemma \ref{commutator}(3) one has
\begin{align*}
    {\bf q}^{\langle {\bf \pi(v)|\pi(u)}\rangle}
    &= \prod\limits_{1\le j,k \le n} ({\bf q}^{\langle  \pi({\bf e}_j)|\pi({\bf e}_k) \rangle})^{v_ju_k}\\
    & =  \prod\limits_{1\le j\le k \le n} ({\bf q}^{\langle  \pi({\bf e}_j)|\pi({\bf e}_k) \rangle})^{v_ju_k}
    \prod\limits_{1\le k<j \le n} ({\bf q}^{\langle  \pi({\bf e}_j)|\pi({\bf e}_k) \rangle})^{v_ju_k}.
\end{align*}
By (2) one has  ${\bf q}^{\langle {\bf e}_j |{\bf e}_k \rangle}{\bf q}^{\langle \pi({\bf e}_j) |\pi({\bf e}_k) \rangle} =
{\bf q}^{\langle {\bf e}_k |{\bf e}_j \rangle}{\bf q}^{\langle \pi({\bf e}_k) |\pi({\bf e}_j) \rangle}, \ \ 1\le j,k\le n.$ Using ${\bf q}^{{\langle {\bf e}_j|{\bf e}_k \rangle}} = 1$ if $1\le j\le k\le n$, one has
\begin{align*}
    {\bf q}^{\langle {\bf \pi(v)|\pi(u)}\rangle} & =  \prod\limits_{1\le j\le k \le n} ({\bf q}^{\langle  \pi({\bf e}_j)|\pi({\bf e}_k) \rangle})^{v_ju_k}
    \prod\limits_{1\le k<j \le n} ({\bf q}^{\langle  \pi({\bf e}_j)|\pi({\bf e}_k) \rangle})^{v_ju_k}\\ & =  \prod\limits_{1\le j\le k \le n} ({\bf q}^{\langle  {\bf e}_j|{\bf e}_k \rangle}{\bf q}^{\langle  \pi({\bf e}_j)|\pi({\bf e}_k) \rangle})^{v_ju_k}
    \prod\limits_{1\le j<k \le n} ({\bf q}^{\langle  \pi({\bf e}_k)|\pi({\bf e}_j) \rangle})^{v_ku_j}\\
    & =  \prod\limits_{1\le j\le k \le n} ({\bf q}^{\langle  {\bf e}_k|{\bf e}_j \rangle}{\bf q}^{\langle  \pi({\bf e}_k)|\pi({\bf e}_j) \rangle})^{v_ju_k}
    \prod\limits_{1\le j<k \le n} ({\bf q}^{\langle  \pi({\bf e}_k)|\pi({\bf e}_j) \rangle})^{v_ku_j}\\
    & =  \prod\limits_{1\le j\le k \le n} ({\bf q}^{\langle  {\bf e}_k|{\bf e}_j \rangle})^{v_ju_k}
    \prod\limits_{1\le j\le k \le n}({\bf q}^{\langle  \pi({\bf e}_k)|\pi({\bf e}_j) \rangle})^{v_ju_k}
    \prod\limits_{1\le j<k \le n} ({\bf q}^{\langle  \pi({\bf e}_k)|\pi({\bf e}_j) \rangle})^{v_ku_j}\\
    & \xlongequal {{\rm Lem. \ref{propeqq}(4)}} {\bf q}^{\langle  {\bf u}|{\bf v} \rangle}
    \prod\limits_{1\le j\le k \le n}({\bf q}^{\langle  \pi({\bf e}_k)|\pi({\bf e}_j) \rangle})^{u_jv_k+u_kv_j}.
\end{align*}

    (4) \ Replacing $\bf u$ by ${\bf a-1-v}$ in (2) and using ${\bf a-1}=\pi(\bf a-1)$, one gets

    \begin{align*}
    h_{\pi({\bf v})}=\frac{{\bf q}^{\langle {\bf a-1-\pi(v)|\pi(v)} \rangle}}{{\bf q}^{\langle {\bf \pi(v)|a-1-\pi(v)} \rangle}}
    \xlongequal{{\rm By\ (1)}} \frac{{\bf q}^{\langle {\bf \pi(a-1-v)|\pi(v)} \rangle}}{{\bf q}^{\langle {\bf \pi(v)|\pi(a-1-v)} \rangle}}
    \xlongequal{{\rm By\ (2)}} \frac{{\bf q}^{\langle {\bf v|a-1-v} \rangle}}{{\bf q}^{\langle {\bf a-1-v|v} \rangle}} =  \frac{1}{h_{\bf v}}.
    \end{align*}

\vskip5pt

\noindent This completes the proof. \end{proof}

\begin{lemma}\label{lemmaI} \ Suppose that $\pi$ is a compatible permutation in $S_n$ with $\pi^2 = \Id$. Put
    $I=\{i \ | \ 1\le i\le n,  \ \pi(i)=i \}$ and $J =\{i \ | \ 1\le i\le n, \ \pi(i)\ne i \} $. Then

\vskip5pt

{\rm (1)} \ \ \ $\{i \ | \  i\in J,\  i<\pi(i) \}\cap \{\pi(i) \  | \  i\in J,\ i<\pi(i) \}=\emptyset;$

\vskip5pt \hskip25pt $J =\{i \ | \  i\in J,\  i<\pi(i) \}\cup \{\pi(i) \  | \  i\in J,\ i<\pi(i) \}.$

\vskip5pt

{\rm (2)} \ \ \ $q_{ij}^2=1$, $\forall \ i,j\in I.$

\vskip5pt

{\rm (3)} \ \ \ $h_{{\bf e}_i} = \prod\limits_{j \in I}  q_{ij}^{a_j-1}, \ \forall \ i\in I.$

\vskip5pt

{\rm (4)} \ \ \ Let $q_\pi$ be as in {\rm Remark \ref{rmkcommutator}}, i.e., $q_\pi = \prod\limits_{1\le j< k\le n}  ({\bf q}^{\langle \pi({{\bf e}_k}) |\pi({{\bf e}_j})\rangle})^{(a_k-1)(a_j-1)}.$ Then
$$q_\pi = \prod\limits_{j, k\in I, \ 1\le j <k \le n}  q_{kj}^{(a_k-1)(a_j-1)} = \pm 1.$$
$($If $I = \emptyset$, then $q_\pi$ is understood to be $1$.$)$\end{lemma}
\begin{proof} {\rm (1)}. \ Assume that otherwise there is an element $i$ in the intersection. Then $i<\pi(i)$ and $i=\pi(j)$ for some $j\in J$ with $j<\pi(j)$. Since $\pi^2=\Id$, one gets a contradiction
    \[
        i=\pi(j)>j = \pi(i).
    \]

For any $j\in J$ with $j\notin \{i \ | \  i\in J,\  i<\pi(i) \}$, then $j>\pi(j)$. By $\pi^2=\Id$ one has $\pi(\pi(j)) = j >\pi(j)$. By definition $j = \pi(\pi(j))\in \{\pi(i) \  | \  i\in J,\ i<\pi(i) \}$. This justifies the second equality.

\vskip5pt

{\rm (2)}.  \ For $i,j\in I$, one has $q_{ij}^2 = q_{ij}q_{ij} = q_{ij}q_{\pi(j)\pi(i)}  = q_{ij}q_{ji}=1.$

\vskip5pt

{\rm (3)}. \  By definition $h_{{\bf e}_i} = \prod\limits_{1\le j \le n}  q_{ij}^{(a_j-1)}$ for $1\le i \le n$. Now let $i\in I$.
For any $j$ with $1\le j\le n$, by  \ $q_{i\pi(j)} =   q_{ji}$ and $a_{\pi(j)}=a_j$ one has $q_{ij}^{a_j-1}q_{i\pi(j)}^{a_{\pi(j)}-1} = (q_{ij}q_{ji})^{a_j-1} = 1$. By (1) one has
\begin{align*}\prod\limits_{j \in J}  q_{ij}^{(a_j-1)}& =\prod\limits_{j \in J, j<\pi(j)}  q_{ij}^{(a_j-1)}\prod\limits_{j=\pi(s), s\in J, s < j}  q_{ij}^{(a_j-1)}
=\prod\limits_{j \in J, j<\pi(j)}q_{ij}^{(a_j-1)}\prod\limits_{s\in J, s <\pi(s)}  q_{i\pi(s)}^{(a_{\pi(s)}-1)}
\\ & = \prod\limits_{j \in J, j<\pi(j)}  q_{ij}^{(a_j-1)}q_{i\pi(j)}^{(a_{\pi(j)}-1)} =1.\end{align*}
  Thus
    \begin{align*}
        h_{{\bf e}_i} = \prod\limits_{1\le j \le n}  q_{ij}^{(a_j-1)} = \prod\limits_{j \in I}  q_{ij}^{(a_j-1)}\prod\limits_{j \in J}  q_{ij}^{(a_j-1)}= \prod\limits_{j \in I}  q_{ij}^{(a_j-1)}, \ \forall \ i\in I.
    \end{align*}

\vskip5pt
{\rm (4)}. \ Since ${\bf q}^{\langle  {\bf e}_k|{\bf e}_j \rangle}=1$ for $k<j$, one has
    \begin{align*}q_\pi & =
        \prod_{1\le j <k \le n} ({\bf q}^{\langle  \pi({\bf e}_k)|\pi({\bf e}_j) \rangle})^{(a_k-1)(a_j-1)}
       \\&  =
        \prod_{\substack{1\le j <k \le n\\ 1\le \pi(j)<\pi(k)\le n}} ({\bf q}^{\langle  \pi({\bf e}_k)|\pi({\bf e}_j) \rangle})^{(a_k-1)(a_j-1)}\\
        & \xlongequal{{\rm Lem. \ref{propeqq}(2)}}
        \prod_{\substack{1\le j <k \le n\\ 1\le \pi(j)<\pi(k)\le n}}  (q_{\pi(j)\pi(k)})^{(a_k-1)(a_j-1)}\\
        & =
        \prod_{\substack{1\le j <k \le n\\ 1\le \pi(j)<\pi(k)\le n}}  q_{kj}^{(a_k-1)(a_j-1)}.
    \end{align*}

\vskip5pt

Assume that $1\le j<k\le n$ and $1\le \pi(j)<\pi(k)\le n$. Then $1\le \pi^2(j)<\pi^2(k)\le n$ since $\pi^2=\Id$. Using the same idea in the proof of (3), we have the following discussions.

\vskip5pt

If $j\in J$ or $k\in J$, then  $q_{kj}^{(a_k-1)(a_j-1)}$  and $q_{\pi(k)\pi(j)}^{(a_{\pi(k)}-1)(a_{\pi(j)}-1)}$ are different two terms occurring in the product $\prod\limits_{\substack{1\le j <k \le n\\ 1\le \pi(j)<\pi(k)\le n}}  q_{kj}^{(a_k-1)(a_j-1)}.$ Using $a_{\pi(i)}=a_i$ and $q_{\pi(i)\pi(j)}=q_{ji}$ for $1\le i,j\le n$, the product of these two terms is
$$q_{kj}^{(a_k-1)(a_j-1)}q_{\pi(k)\pi(j)}^{(a_{\pi(k)}-1)(a_{\pi(j)}-1)} = (q_{kj}q_{jk})^{(a_k-1)(a_j-1)} = 1.$$ Thus, in the product
$\prod\limits_{\substack{1\le j <k \le n\\ 1\le \pi(j)<\pi(k)\le n}}  q_{kj}^{(a_k-1)(a_j-1)}$, we only need to consider the product of those terms
$(q_{kj})^{(a_k-1)(a_j-1)}$ with $1\le j <k \le n, \ 1\le \pi(j)<\pi(k)\le n$ and $j\in I, \ k\in I$. Note that if $1\le j <k \le n$ and $j\in I, \ k\in I$, then
$1\le \pi(j)<\pi(k)\le n$  holds automatically. It follows that
   $$q_\pi =
        \prod_{\substack{1\le j <k \le n\\ 1\le \pi(j)<\pi(k)\le n}}  (q_{kj})^{(a_k-1)(a_j-1)} =
        \prod_{j, k \in I, j <k}  (q_{kj})^{(a_k-1)(a_j-1)}.$$
By (2), one gets $q_\pi=\pm 1$. This completes the proof.
\end{proof}

\section{\bf Bi-Frobenius quantum complete intersections}

For a bi-Frobenius algebra $(A, \phi, t, S)$, the Frobenius form $\phi$, the right integral $t$, Nakayama automorphisms, the antipode $S$, the right modular function $\alpha$,
and the right modular element $\mathfrak{a}$, provide fundamental structural information. The following result collects these for bi-Frobenius quantum complete intersections,
where (i), (ii) and (iii) below are known.

\begin{proposition}\label{thenecessity}  Let $(A,  \ \phi,  \ t, \ S)$ be a bi-Frobenius algebra structure on $A=A({\bf q,a})$.
Then

\vskip5pt

{\rm (i)} \ \  $\phi = x_{\bf a-1}^*\leftharpoonup z$ for some $z\in U(A);$ and $\phi = \sum\limits_{\mathbf{v}\in V} c_\mathbf{v}x_\mathbf{v}^*$  with \ $c_\mathbf{v} \in \mathbb K$  \ and \  $c_{\mathbf{a-1}} \neq 0.$

\vskip5pt

{\rm (ii)} \ \ $\alpha = \varepsilon;$ and $\varepsilon(x_\mathbf{v})=\delta_{\mathbf{v,0}}\ \ \text{for} \ \  \mathbf{v}\in V. $

\vskip5pt

{\rm (iii)} \ \ $t = cx_\mathbf{a-1}$ with $c\in\mathbb K-\{0\};$ and $A$ is unimodular with the space of right integrals $\mathbb Kx_{\bf a-1}$.

\vskip5pt

{\rm (iv)} \ \ $\mathcal{N}^2=\Id$, i.e., $h^2_{{\bf e}_i} = 1, \ 1\le i\le n$, where $\mathcal{N}$ is the canonical Nakayama automorphism of $A$.

\vskip5pt

{\rm (v)} \ \   If $S$ is a graded map, then $S^2 = \mathcal N$ and $S^4 = \Id$.

\vskip5pt

{\rm (vi)} \ \ If ${\rm Char} \mathbb K = 0,$ then $\mathfrak{a} = 1, \ \mathcal{N}_\phi^2=\Id$ and $S^4 = \Id$.
\end{proposition}
\begin{proof}  For (i) see Lemma \ref{frobeniushomomorphism}.  The assertion {\rm (ii)} follows from the fact that algebra homomorphism $A\longrightarrow \mathbb K$ is unique, for a local finite-dimensional $\mathbb K$-algebra $A$.
If fact, the right modular function $\alpha: A\longrightarrow \mathbb K$ is an algebra homomorphism (cf. Lemma \ref{lemmanak}(1)). Since $A$ is local, ${\rm Ker} \alpha = \rad A$ and $\alpha (1_A) = 1_A$.
While $\varepsilon: A \longrightarrow \mathbb K$ is also an algebra homomorphism, thus  $\alpha = \varepsilon$.

\vskip5pt

(iii) \ Write $t = \sum\limits_{\mathbf{v}\in V} c_{\mathbf{v}}x_\mathbf{v}$ with each $c_{\mathbf{v}}\in\mathbb K$. Since $t$ is a right integral of $A$ (cf. Lemma \ref{dt}(2)(iii)),
one has $t x=t \varepsilon(x) = x^*_{\mathbf{0}}(x) t$ for all \ $x\in A$. Thus, if
$\mathbf v < \mathbf{a-1}$, then
$$0 = x^*_{\mathbf{0}}(x_\mathbf{a-1-v}) t = tx_\mathbf{a-1-v} = \sum\limits_{\mathbf{u}\in V} c_{\mathbf{u}}x_\mathbf{u}x_\mathbf{a-1-v}
= \sum\limits_{\mathbf{u}\in V} c_{\mathbf{u}}{\bf q}^{\langle \mathbf u | \mathbf {a-1-v} \rangle}x_\mathbf{a-1+u-v}.$$
In particular, $c_{\mathbf{v}} {\bf q}^{\langle \mathbf v | \mathbf {a-1-v} \rangle}x_\mathbf{a-1} = 0$, which implies $c_{\mathbf{v}} = 0$, for all $\mathbf v < \mathbf{a-1}$. Thus $t = cx_\mathbf{a-1}$ with $c\in\mathbb K$. Since $A = t \leftharpoonup A^*$, $c\ne 0$.
This also proves that the space of right integrals is $\mathbb Kx_{\bf a-1}$.

\vskip5pt

Similarly, the space of left integrals is $\mathbb Kx_{\bf a-1}$.  Thus $A$ is unimodular.

\vskip5pt

(iv) \ By (ii), $\alpha = \varepsilon$. Thus for $x\in A$ one has
$$x\leftharpoonup \alpha = x\leftharpoonup \varepsilon = \sum \varepsilon (x_1)x_2 = (\varepsilon\otimes \Id)\Delta (x) = x$$ and  \
$\alpha\rightharpoonup x = \varepsilon\rightharpoonup x = \sum x_1\varepsilon (x_2) = x.$

\vskip5pt

Using Lemma \ref{lemmanak}(3) twice one gets
$$\mathcal{N}_\phi^2(x_i)= \mathcal{N}_\phi(\mathfrak{a}^{-1}S^2(\alpha\rightharpoonup x_i)\mathfrak{a})  = \mathcal{N}_\phi(\mathfrak{a}^{-1}S^2(x_i)\mathfrak{a}) = \mathfrak{a}^{-1}S^2( \mathfrak{a}^{-1}S^2( x_i)\mathfrak{a})\mathfrak{a}.$$ By Lemma \ref{lemmanak}(2) one has
$S^2(\mathfrak{a}^{-1}) = \mathfrak{a}^{-1}$ \ and \ $S^2(\mathfrak{a}) = \mathfrak{a}$. Thus $$\mathcal{N}_\phi^2(x_i) = \mathfrak{a}^{-1}S^2(\mathfrak{a}^{-1}S^2( x_i)\mathfrak{a})\mathfrak{a}
         = \mathfrak{a}^{-1}S^2(\mathfrak{a}^{-1})S^4( x_i)S^2(\mathfrak{a})\mathfrak{a} = \mathfrak{a}^{-2} S^4(x_i)\mathfrak{a}^2.$$

Note that the identity of algebra $A^*$ is $\varepsilon$. Thus $\alpha^{-1} = \varepsilon^{-1} = \varepsilon$.
It follows from Lemma \ref{lemmanak}(4) that
\begin{align*}
        \mathcal{N}_\phi^2(x_i) & = \mathfrak{a}^{-2} S^4(x_i)\mathfrak{a}^2 = \mathfrak{a}^{-1} (\alpha^{-1}\rightharpoonup x_i\leftharpoonup \alpha) \mathfrak{a} \\
        & = \mathfrak{a}^{-1} (\varepsilon \rightharpoonup x_i\leftharpoonup \varepsilon) \mathfrak{a} = \mathfrak{a}^{-1} x_i\mathfrak{a}.
    \end{align*}
Since $A$ is local and $\mathfrak{a}$ is an invertible element of $A$,  $\mathfrak{a}  = c1_A+r$ with $c\in\mathbb{K}-\{0\}, \ r\in \rad A$.
Since $\mathfrak{a}$ is a group-like element, thus $\varepsilon({\mathfrak a})= 1$ and ${\mathfrak a} = 1+r$. Thus
$$\mathcal{N}_\phi^2(x_i) = \mathfrak{a}^{-1} x_i\mathfrak{a} = x_i + \Sigma_i, \ \forall \ 1\le i\le n$$
where $\Sigma_i\in \rad^2 A$.

\vskip5pt

On the other hand, by [B, Lemma 3.1],  $\mathcal{N}(x_i) = h_{{\bf e}_i}x_i$ for $1\le i\le n$. Since $\phi = x_{\bf a-1}^*\leftharpoonup z$ for some $z \in U(A)$, it follows from Lemma \ref{nakayama}
that $$\mathcal{N}_\phi(x_i) = z\mathcal{N}(x_i)z^{-1} = zh_{{\bf e}_i}x_iz^{-1}  = h_{{\bf e}_i}x_i + \Sigma'_i$$
with $\Sigma'_i\in \rad^2 A$. Since $\mathcal N_\phi$ is an algebra automorphism, $\mathcal N_\phi(\rad^2 A) \subseteq \rad^2 A$. Thus
$$\mathcal{N}_\phi^2(x_i) = \mathcal{N}_\phi(h_{{\bf e}_i}x_i + \Sigma'_i) =  h^2_{{\bf e}_i}x_i + \Sigma''_i$$
with $\Sigma''_i\in \rad^2 A$. Comparing with $\mathcal{N}_\phi^2(x_i) = x_i + \Sigma_i$ one gets
    \[
        h_{{\bf e}_i}^2=1, \ \forall \ 1\le i\le n.
    \]
Thus  $\mathcal{N}^2(x_i) = h_{{\bf e}_i}^2 x_i = x_i$ for each $i$.  Since $\{x_i \ | \ 1\le i\le n\}$ is  a set of generators of $A$, one gets $\mathcal{N}^2 = \Id.$

\vskip5pt

(v) \ Assume that $S$ is a graded map. By Lemma \ref{lemmanak}(4) we have that
$$S^4(x_i) = \mathfrak{a} (\alpha^{-1}\rightharpoonup x_i\leftharpoonup \alpha) \mathfrak{a}^{-1}
= \mathfrak{a} (\alpha\rightharpoonup x_i\leftharpoonup \alpha) \mathfrak{a}^{-1} = \mathfrak{a} x_i\mathfrak{a}^{-1} = x_i + \Sigma_i, \ \forall \ 1\le i\le n$$
with $\Sigma_i\in \rad^2 A$. Since $S$ is graded, $S^4(x_i) = x_i \ (1\le i\le n)$. Hence $S^4 = \Id$.

\vskip5pt

Let $\phi = x_{\bf a-1}^*\leftharpoonup z$ with $z\in U(A)$. Then $\mathcal N_\phi(x_{\mathbf v}) = \mathcal N(zx_{\mathbf v}z^{-1}).$
By Lemma \ref{lemmanak}(3) and the assertion (ii) above one has
\begin{align*}S^2(x_{\mathbf v}) & = S^2(\varepsilon\rightharpoonup x_{\mathbf v}) = S^2(\alpha \rightharpoonup x_{\mathbf v})=
\mathfrak{a}\mathcal{N}_\phi(x_{\mathbf v}) \mathfrak{a}^{-1}=
\mathfrak{a}\mathcal N(zx_{\mathbf v}z^{-1}) \mathfrak{a}^{-1}. \\ & = \mathfrak{a}\mathcal N(z) \mathcal N(x_{\mathbf v}) \mathcal N (z^{-1}) \mathfrak{a}^{-1}
= h_{\mathbf v}\mathfrak{a}\mathcal N(z) x_{\mathbf v} (\mathfrak{a}\mathcal N (z))^{-1}.
\end{align*}
Note that $\mathfrak{a}\mathcal N(z) = c1_A + r\in U(A)$ with $(\mathfrak{a}\mathcal N (z))^{-1} = c^{-1}1_A + r'$, where $r, r'\in {\rad}(A)$.
Thus $$S^2(x_{\mathbf v})  = h_{\mathbf v} x_{\mathbf v} + \Sigma_{\mathbf v}$$ with $\Sigma_{\mathbf v}\in {\rad}^{|x_{\mathbf v}|+1}(A)$. Since $S$ is graded, it follows that
$\Sigma_{\mathbf v} = 0$ and hence $S^2(x_{\mathbf v})  = h_{\mathbf v} x_{\mathbf v} = \mathcal N(x_{\mathbf v}), \ \forall \ {\bf v}\in V$.
Thus $S^2 = \mathcal N$.

\vskip5pt

(vi) \ Note that $\mathfrak{a}^{n}=1$ for some positive integer $n$. In fact,  ${\mathfrak a}^m$ is a group-like element for $m\in\mathbb{Z}$ (cf. Lemma \ref{lemmanak}(1)). Thus  $\varepsilon({\mathfrak a})= 1$ and ${\mathfrak a} = 1+r$ for some $r\in \rad(A)$.
    Since distinct group-like elements are linearly independent, there are positive integers $m_1\ne m_2$ such that $\mathfrak{a}^{m_1}=\mathfrak{a}^{m_2}$.
Since  $\mathfrak{a}$ is invertible,  $\mathfrak{a}^{n}=1$ for some positive integer $n$.  Thus
    \[0=
    (1+r)^n-1=\sum_{i=0}^{n}\binom{n}{i}r^i-1=\sum_{i=1}^{n}\binom{n}{i}r^i =r(n 1_A +r'),
    \]
    where $r'=\sum\limits_{1\le i\le n}\binom{n}{i}r^{i-1}\in \rad(A)$ is nilpotent. Since ${\rm char}\mathbb{K}=0$, $n 1_A +r'$ is invertible, thus  $r=0$ and $\mathfrak{a}=1$.
Then from the proof of (iv) one gets \ $\mathcal{N}_\phi^2=\Id$ and $S^4 = \Id$.
\end{proof}
\begin{rmk}\label{N2} The fact that the canonical Nakayama automorphism $\mathcal N$ is of order $2$, for those quantum complete intersections $A$ which admit a bi-Frobenius algebra structure,
is a quite special property. Indeed, even for a quantum complete intersection, $\mathcal N_\phi$ is not in general of order $2$ for $\phi\ne x_{{\bf a-1}}^*$.
See {\rm Example \ref{exmnonsymmtric}}.
\end{rmk}

\section{\bf Quantum complete intersections with permutation antipodes}

Since actions of antipodes of all the known bi-Frobenius algebra structures on quantum complete intersections
are given by scalar multiplications  on the standard basis $\mathcal B = \{x_{\mathbf{v}} \ | \ \mathbf{v}\in V\}$ up to a permutation, it is natural to consider the following notion.

\begin{definition}\label{perS} \  Let \ $(A, \ \phi, \ t, \ S)$ be a  bi-Frobenius algebra structure on quantum complete intersection $A=A({\bf q,a})$.
Then $S$ is said to be a permutation antipode,
provided that there is a permutation  $\pi$ on $V$ and $c_{\mathbf v}\in \mathbb K, \ \forall \ \mathbf v\in V$,  such that $S(x_{\mathbf v}) = c_{\mathbf v} x_{\pi(\mathbf v)}, \ \forall \ \mathbf v\in V.$ \end{definition}

\vskip5pt

If this is the case, we also call $(S, \pi)$ {\it the antipode with permutation $\pi$}.

\vskip5pt

\subsection{Bi-Frobenius quantum complete intersections with permutation antipode} The following result gives properties of bi-Frobenius quantum complete intersections $A=A({\bf q,a})$ with permutation antipode; in the next
section, by construction, we will see they are also sufficient for   $A$ to become a bi-Frobenius algebra with permutation antipode.

\vskip5pt

Recall that $\pi\in S_n$ is a compatible permutation,
provided that $a_{\pi(i)}=a_i$ and $q_{\pi(i)\pi(j)}=q_{ji}$ for $1\le i,j\le n$ (cf. Definition \ref{compatible}).

\begin{theorem}\label{thmperS} If \ $A=A({\bf q,a})$ admits a bi-Frobenius algebra structure with permutation antipode $(S, \pi)$.
Then

\vskip5pt

{\rm (i)} \  $S(x_{\bf {a-1}}) = x_{\bf{a-1}};$ \ \ $\pi(\bf{a-1}) = \bf{a-1};$ \ \ $\pi^2 = \Id;$ \ \ $S$ is a graded map, \ \ $S^2 = \mathcal N$ and \ \ $S^4 = \Id$.

\vskip5pt

{\rm (ii)} \  $\pi$ induces a compatible permutation in $S_n$, still denoted by $\pi$, such that
$$\pi({\bf v}) =(v_{\pi^{-1}(1)},\cdots, v_{\pi^{-1}(n)}), \ \ \ \ x_{\pi(\mathbf v)} = x_1^{v_{\pi^{-1}(1)}} \cdots x_n^{v_{\pi^{-1}(n)}}, \ \ \forall \ {\bf v}\in V;$$ and
that $S(x_i) = c_ix_{\pi(i)},$ where each $c_i\in \mathbb K,$ satisfying

$$c_ic_{\pi(i)}=h_{{\bf e}_i}, \ \forall \ 1\le i\le n; \ \ \ \ \ \ \ q_\pi \prod\limits_{1\le i\le n} c_i^{a_i-1} =1
$$
where $q_\pi$ is defined in {\rm Remark \ref{rmkcommutator}}, i.e.,  $q_\pi = \prod\limits_{1\le j <k \le n} ({\bf q}^{\langle  \pi({\bf e}_k)|\pi({\bf e}_j) \rangle})^{(a_k-1)(a_j-1)}.$\end{theorem}

\begin{proof} \  By Proposition \ref{thenecessity}(iii), $t = cx_{\bf {a-1}}$ with $c\in\mathbb K-\{0\}$. By Lemma \ref{dt}(2)(vi) and Proposition \ref{thenecessity}(iii) one has $S(t) = t$. Thus
$$cx_{\bf {a-1}} = t = S(t) = S(cx_{\bf {a-1}}) = c S(x_{\bf {a-1}}).$$ Hence $S(x_{\bf {a-1}}) = x_{\bf{a-1}}.$ Therefore
$$x_{\bf{a-1}} = S(x_{\bf {a-1}}) = c_{\bf{a-1}} x_{\pi(\bf {a-1})}.$$ Thus $x_{\pi(\bf {a-1})} = x_{\bf{a-1}}$ and $c_{\bf{a-1}} = 1$, so $\pi(\bf{a-1}) = \bf{a-1}.$

\vskip5pt

By \begin{align*}0 & \ne S(x_{\mathbf {a-1}}) = S(x_1^{a_1-1}\cdots x_n^{a_n-1}) = S(x_n)^{a_n-1}\cdots S(x_1)^{a_1-1}
\\ & = S(x_{\mathbf {e}_n})^{a_n-1}\cdots S(x_{\mathbf {e}_1})^{a_1-1}  = c x_{\pi({\bf e}_n)}^{a_n-1}\cdots x_{\pi({\bf e}_1)}^{a_1-1}\end{align*}
with $c\in \mathbb K$, we see that each degree $|x_{\pi({\bf e}_i)}| = 1$, since otherwise one has
$$|x_{\pi({\bf e}_n)}^{a_n-1}\cdots x_{\pi({\bf e}_1)}^{a_1-1}| > \sum\limits_{i=1}^n (a_i - 1),$$ and then
$S(x_{\mathbf {a-1}}) = c x_{\pi({\bf e}_n)}^{a_n-1}\cdots x_{\pi({\bf e}_1)}^{a_1-1} = 0, $ a contradiction!
Thus, $\pi$ induces a permutation on the subset $\{{\bf e}_1, \ \cdots, \ {\bf e}_n\}$ of $V$, still denoted by $\pi$, namely, $\pi\in S_n$, such that
$\pi({\bf e}_i) = {\bf e}_{\pi(i)}, \ 1\le i\le n.$

\vskip5pt

Using the defining relation $x_jx_i = q_{ij}x_ix_j$, one can write $x_{\pi(n)}^{v_n}\cdots x_{\pi(1)}^{v_1} = c''x_1^{v_{\pi^{-1}(1)}} \cdots x_n^{v_{\pi^{-1}(n)}}$ with $c''\in\mathbb K$. Therefore, for ${\bf v}\in V$ one has
\begin{align*} c_{\mathbf v} x_{\pi(\mathbf v)} & = S(x_{\mathbf v}) = S(x_1^{v_1}\cdots x_n^{v_n}) = S(x_{\mathbf {e}_n})^{v_n}\cdots S(x_{\mathbf {e}_1})^{v_1}
\\ & = c' x_{\pi({\bf e}_n)}^{v_n}\cdots x_{\pi({\bf e}_1)}^{v_1}= c' x_{{\bf e}_{\pi(n)}}^{v_n}\cdots x_{{\bf e}_{\pi(1)}}^{v_1}\\ & =
c' x_{\pi(n)}^{v_n}\cdots x_{\pi(1)}^{v_1} = c'c''x_1^{v_{\pi^{-1}(1)}} \cdots x_n^{v_{\pi^{-1}(n)}}\end{align*}
with $c'c''\in\mathbb K.$ Thus $x_{\pi(\mathbf v)} = x_1^{v_{\pi^{-1}(1)}} \cdots x_n^{v_{\pi^{-1}(n)}}$ and hence $\pi({\bf v}) =(v_{\pi^{-1}(1)},\cdots, v_{\pi^{-1}(n)})$.

\vskip5pt

In particular, one has $$S(x_i) = S(x_{\mathbf {e}_i}) = c_{\mathbf {e}_i} x_{\pi(\mathbf {e}_i)} = c_{\mathbf {e}_i} x_{{\bf e}_{\pi(i)}} =  c_ix_{\pi(i)}, \ 1\le i\le n$$
with $c_i = c_{\mathbf {e}_i}$. By the argument above one sees that
$$x_{\pi({\bf e}_i)}^{a_i-1} = x_{\pi(i)}^{a_i-1}\ne 0, \ \ 1\le i\le n.$$
Thus $a_{\pi(i)} \ge a_i, \ \ 1\le i\le n$. Since $\pi\in S_n$ is of finite order, it follows that $a_{\pi(i)} = a_i, \ 1\le i\le n.$

\vskip10pt

For  $1\le i, j\le n$, one has
\begin{align*}
S(x_i)S(x_j)&= S(x_jx_i) = q_{ij}S(x_ix_j) = q_{ij}S(x_j)S(x_i)=c_ic_jq_{ij}x_{ \pi(j)}x_{ \pi(i)}\\
& = c_ic_jq_{ij}q_{\pi(i)\pi(j)}x_{ \pi(i)}x_{ \pi(j)}=q_{ij}q_{\pi(i)\pi(j)}S(x_i)S(x_j).
\end{align*}
Since $q_{ij}q_{ji}=1, \  \ \forall \ 1\le i,j\le n,$ one gets $q_{\pi(i)\pi(j)}=q_{ji}, \ \ \forall \ 1\le i,j\le n.$ Thus, $\pi\in S_n$ is a compatible permutation.

\vskip5pt

By $S(x_i) = c_ix_{\pi(i)}, \ \ 1\le i\le n,$ we know that \ $S$ is a graded map, and hence \ $S^2 = \mathcal N$ and \ $S^4 = \Id$, by Proposition \ref{thenecessity}(v).

\vskip5pt

By $S(x_{\mathbf v}) = c_{\mathbf v} x_{\pi(\mathbf v)}$ and $S^2 = \mathcal N$ one has \ $S^2(x_{\mathbf v})  = c_{\mathbf v} c_{\pi(\mathbf v)} x_{\pi^2(\mathbf v)}, \ \forall \ \mathbf v\in V,$ and
$$S^2(x_{\mathbf v})  = \mathcal N(x_{\mathbf v}) = h_{\mathbf v}x_{\mathbf v}.$$
It follows that
$c_{\mathbf v} c_{\pi(\mathbf v)} = h_{\mathbf v}, \ \ x_{\pi^2(\mathbf v)} = x_{\mathbf v}, \ \forall \ \mathbf v\in V.$ Thus $c_ic_{\pi(i)}=h_{{\bf e}_i}, \ 1\le i\le n$ and $\pi^2(\mathbf v) = \mathbf v, \ \forall \ \mathbf v\in V,$ i.e., $\pi^2 = \Id$.

\vskip5pt

Finally,  since $S(x_{\bf a-1})=x_{\bf a-1}$, one has
    \begin{align*}
        x_{\bf a-1}&=S(x_{\bf a-1}) = S(x_n)^{a_n-1}S(x_{n-1})^{a_{n-1}-1}\cdots S(x_1)^{a_1-1}\\
        & =  \prod_{1\le i\le n} c_i^{a_i-1} x_{\pi ({\bf e}_n)}^{a_n-1}x_{\pi ({\bf e}_{n-1})}^{a_{n-1}-1}\cdots x_{\pi ({\bf e}_1)}^{a_1-1}
    \end{align*}
By Remark \ref{rmkcommutator} one has
$$x_{\pi ({\bf e}_n)}^{a_n-1}x_{\pi ({\bf e}_{n-1})}^{a_{n-1}-1}\cdots x_{\pi ({\bf e}_1)}^{a_1-1}  = q_\pi x_{\pi({\bf a-1})} = q_\pi x_{\bf a-1}.$$
It follows that \ $q_\pi \prod\limits\limits_{1\le i\le n} c_i^{a_i-1}=1$.
This completes the proof. \end{proof}

\vskip5pt

\subsection{Equivalent descriptions of a permutation antipode}
\begin{proposition}\label{2perS} \ Let \ $(A, \ \phi, \ t, \ S)$ be a  bi-Frobenius algebra structure on $A=A({\bf q,a})$.
Then the following statements are equivalent$:$
\vskip5pt
{\rm (1)} \ $S$ is a permutation antipode$;$
\vskip5pt
{\rm (2)} \ There is  $\pi\in S_n$ and  $c_i\in \mathbb K, \ 1\le i\le n$,  such that
\ $S(x_i) = c_ix_{\pi(i)}, \ \ 1\le i\le n;$
\vskip5pt
{\rm (3)} \ There is a compatible permutation $\pi\in S_n$ and  $c_i\in \mathbb K, \ 1\le i\le n$,  such that
\ $S(x_i) = c_ix_{\pi(i)}, \ \ 1\le i\le n.$
\end{proposition}
\begin{proof} The implication (1)$\Longrightarrow$ (3) follows from Theorem \ref{thmperS};  (3)$\Longrightarrow$ (2) is trivial. We show (2)$\Longrightarrow$(1).
Assume that there is  $\pi\in S_n$ and  $c_i\in \mathbb K, \ 1\le i\le n$,  such that
\ $S(x_i) = c_ix_{\pi(i)}, \  1\le i\le n.$ As in the proof of Theorem \ref{thmperS} one has  $a_{\pi(i)} = a_i, \  1\le i\le n.$ In fact,
Since $S$ is bijective, one has
$$0 \ne S(x_i)^{a_i-1} = c_i^{a_i-1} x_{\pi(i)}^{a_i-1}, \  1\le i\le n.$$
Thus $a_{\pi(i)} \ge a_i, \  1\le i\le n$. Since $\pi\in S_n$ is of finite order,  $a_{\pi(i)} = a_i, \ 1\le i\le n.$

\vskip5pt

Define a permutation on $V$, still denoted by $\pi$, as follows:
$$\pi({\bf v}) =(v_{\pi^{-1}(1)}, \cdots, v_{\pi^{-1}(n)}), \ \ \forall \ {\bf v} = (v_1, \cdots, v_n)\in V.$$
Since $v_i\le a_i-1$ and \ $a_{\pi(i)} = a_i$ for \ $1\le i\le n$, it follows that
$$v_{\pi^{-1}(i)} \le a_{\pi^{-1}(i)} -1 = a_i -1, \ 1\le i\le n.$$
This implies that \ $\pi({\bf v})\in V$, i.e.,  $\pi$ is a well-defined permutation on $V$.
Thus, for ${\bf v}\in V$ one has
\begin{align*}S(x_{\mathbf v}) & = S(x_1^{v_1}\cdots x_n^{v_n}) = S(x_n)^{v_n}\cdots S(x_1)^{v_1}
\\ & = c_1\cdots c_n x_{\pi(n)}^{v_n}\cdots x_{\pi(1)}^{v_1} = c_1\cdots c_nc'x_1^{v_{\pi^{-1}(1)}} \cdots x_n^{v_{\pi^{-1}(n)}}\\& = c_{\mathbf v} x_{\pi(\mathbf v)}\end{align*}
where $c_{\mathbf v} = c_1\cdots c_nc'\in\mathbb K$. By definition $S$ is a permutation antipode.
\end{proof}

\subsection {\bf The case of $a_i\ge 3$}

\begin{proposition} \label {a>2}\ For $A=A({\bf q,a})$ with $a_i\ge 3$ and $q_{ij}\ne 1$ for $1\le i\ne j\le n$, if $A$ admits a bi-Frobenius algebra structure such that $S$ a graded map,
    then $S$ is necessarily a permutation antipode.
\end{proposition}
\begin{proof}
    Since $S$ is graded, one can assume that $S(x_i) =\sum\limits_{1\le j\le n} m_{ij}x_j, \ \forall \ 1\le i\le n.$ Since $S$ is bijective,  $M=(m_{ij})$ is an invertible matrix. By $x_jx_i-q_{ij}x_ix_j=0$ one has
$$S(x_i)S(x_j) = q_{ij}S(x_j)S(x_i),\ \forall \ 1\le i,j\le n.$$
That is $$\sum\limits_{1\le s, t \le n} m_{is}m_{jt} x_sx_t= q_{ij}\sum\limits_{1\le s, t \le n} m_{js}m_{it}x_sx_t,\ \forall \ 1\le i, j\le n.$$
    Since $a_i\ge 3$ for $1\le i\le n$, it follows that $x_s^2\ne 0, \ 1\le s\le n$.
    Comparing the coefficient of $x_s^2$ in the both sides of the equality above, one gets
    $$m_{is}m_{js}=q_{ij}m_{is}m_{js}, \ \forall \ 1\le i, j, s\le n.$$ Since by assumption $q_{ij}\ne 1$ for $1\le i\ne j\le n$, one gets
    \[
    m_{is}m_{js}=0, \ \forall \ 1\le i,j,s\le n,\ i\ne j.
    \]
This implies that in the $s$-th ($1\le s\le n$) column of $M$  there is at most one nonzero element.
Since  $M$ is invertible, in each column of $M$ there is a unique nonzero element; and hence in each row of $M$ there is a unique nonzero element.
Thus,  there is $\pi\in S_n$ such that $m_{i\pi(i)}\ne 0, \ 1\le i\le n,$ and $S(x_i) = m_{i\pi(i)}x_{\pi(i)}, \ 1\le i\le n$. By Proposition \ref{2perS}, $S$ is
a permutation antipode.
\end{proof}

\begin{rmk} \label{rmka>2} {\rm Proposition \ref{a>2}} actually shows that when $a_i\ge 3$ and $q_{ij}\ne 1$ for $1\le i\ne j\le n$, any graded anti-automorphism $S$ of $A$ has the property  $S(x_i) = c_ix_{\pi(i)}$ for some  $c_i\in\mathbb{K},$ $1\le i\le n$, and some compatible permutation $\pi\in S_n$. The similar result for graded automorphisms of quantum affine $n$-space has been obtained by
{\rm S. Ceken, J. H. Palmieri, Y. H. Wang, J. J. Zhang} in {\rm \cite[Proposition 3.9]{CPWZ2016}}.
The similar considerations have also been discussed for affine connected graded algebras and some quotients of quantum affine $n$-space e.g. in {\rm \cite{Ga2013}, \cite{BZ2017}}.
\end{rmk}

\section{\bf Construction of bi-Frobenius quantum complete intersections}

We will construct a class of bi-Frobenius algebra structures
on $A = A({\bf q, a})$ with permutation antipode.
A main difficulty   is the construction of comultiplications such that $A$ becomes
a bi-Frobenius algebra.

\vskip5pt

Recall that $\pi\in S_n$ is a compatible permutation,
if $a_{\pi(i)}=a_i$ and $q_{\pi(i)\pi(j)}=q_{ji}$ for $1\le i,j\le n$.
The following main result gives a sufficient and necessary condition such that $A$ becomes
a bi-Frobenius algebra with permutation antipode.

\begin{theorem}\label{mainthm}  Quantum complete intersection $A({\bf q,a })$ admits a bi-Frobenius algebra structure with permutation antipode
if and only if there is a compatible permutation $\pi\in S_n$ with $\pi^2 = \Id$ and there are $c_i\in\mathbb K, \ 1\le i\le n$,  such that
$$c_ic_{\pi(i)}=h_{{\bf e}_i},  \ \ \ \ \ \ \ q_\pi \prod\limits_{1\le i\le n} c_i^{a_i-1} =1$$
where $q_\pi = \prod\limits_{1\le j <k \le n} ({\bf q}^{\langle  \pi({\bf e}_k)|\pi({\bf e}_j) \rangle})^{(a_k-1)(a_j-1)}.$
\vskip5pt

If this is the case, then  $(A, \ x_{\bf a-1}^*, \ x_{\bf a-1}, \ S)$ is a bi-Frobenius algebra,
where  the comultiplication $\Delta$ is given by $(6.3)$ and $S$ is given by $(\ref{equationS})$ below. In particular, $S(x_i) = c_ix_{\pi(i)}, \ 1\le i\le n$.
\end{theorem}

\begin{rmk} If $\mathbb K$ contains $\sqrt{-1}$ and $q_{ij} = \pm 1, \ 1\le i, j\le n$, then one can take $\pi = \Id$ and choose some $c_i\in\mathbb K, \ 1\le i\le n$, such that all the  conditions
in {\rm Theorem \ref{mainthm}} are satisfied. This special case has been treated in {\rm [JZ, Theorem 3.2]}.
\end{rmk}

To prove Theorem \ref{mainthm}, we need some preparations.

\subsection{\bf Structure constants $g_{\bf a-1-v, \pi(v)}$}
To get coalgebra structures on $A=  A({\bf q,a})$, one needs to define elements $g_{\bf a-1-v, \pi(v)}\in\mathbb K$ for each ${\bf v}\in V$,
which will be taken as the structure constants.

\vskip5pt

{\bf Setting up:} \ From now on, we consider  $A =  A({\bf q,a})$ such that
there is a compatible permutation $\pi\in S_n$ with $\pi^2 = \Id$, and there are $c_i\in\mathbb K, \ 1\le i\le n$,  such that
$$c_ic_{\pi(i)}=h_{{\bf e}_i}, \ 1\le i, j\le n;  \ \ \ \ \ \ q_\pi\prod\limits_{1\le i\le n} c_i^{a_i-1} = 1.$$

\vskip5pt

Since $a_{\pi(i)}=a_{i}, \  1\le i\le n,$ it follows that if ${\bf  v}= (v_1, \cdots, v_n)\in V$ then
$$0\le v_{\pi^{-1}(i)} \le a_{\pi^{-1}(i)}-1 = a_i -1, \ 1\le i\le n.$$
Thus $\pi(V)\subseteq V$, and hence $\pi$ induces an action on $V$ as follows, still denoted by $\pi$,
$$\pi({\bf v}) =(v_{\pi^{-1}(1)},\cdots, v_{\pi^{-1}(n)}), \ \ \forall \ {\bf v} = (v_1, \cdots, v_n)\in V.$$
Thus $|x_{\pi({\bf v})}| = |x_{\bf v}|, \ \forall \ {\bf v}\in V$, and $\pi({\bf 0}) = {\bf 0}, \ \pi({\bf e}_i) = {\bf e}_{\pi(i)}, \ 1\le i\le n.$ By Lemma \ref{propusefuleq2}(1) one has $\pi({\bf a-1}) = {\bf a-1}.$
We will use this induced action throughout, without further explanations.

\vskip5pt

For ${\bf v}\in V$, define $g_{\bf a-1-v, \pi(v)}\in\mathbb K$ as follows:
\begin{align}\label{equationg}
    g_{\bf a-1-v,\pi(v)} =\frac{1}{{\bf q}^{\langle {\bf a-1-v|v} \rangle}}
    \prod_{1\le i\le n} c_{i}^{v_i}\prod_{1\le j <k \le n} ({\bf q}^{\langle  \pi({\bf e}_k)|\pi({\bf e}_j) \rangle})^{v_jv_k}, \ \forall \ {\bf v}\in V.
\end{align}
Then one has $g_{{\bf a-1}-{\bf e}_i,\pi({\bf e}_i)} = \frac{c_{i}}{{\bf q}^{\langle{\bf a-1}-{\bf e}_i|{\bf e}_i \rangle}}, \ \forall \ 1\le i\le n.$

\vskip5pt

By $\pi^2 = \Id$ one has
\begin{align*}
    g_{\bf a-1-\pi(v),v}= \frac{1}{{\bf q}^{\langle {\bf a-1-\pi(v)|\pi(v)} \rangle}}\prod_{1\le i\le n} c_{i}^{v_{\pi(i)}}\prod_{1\le j <k \le n} ({\bf q}^{\langle  \pi({\bf e}_k)|\pi({\bf e}_j) \rangle})^{v_{\pi(j)}v_{\pi(k)}}
    , \ \forall \ {\bf v}\in V.
\end{align*}
For later use we need an alternative expression of $g_{\bf a-1-\pi(v),v}$, for which we now deduce  as follows.
Note that ${\bf q}^{\langle \pi({\bf e}_k)|\pi({\bf e}_j) \rangle}=1$ if $1\le \pi(k)\le \pi(j)\le n$. Thus

\begin{align*}
    \prod_{1\le j <k \le n} ({\bf q}^{\langle  \pi({\bf e}_k)|\pi({\bf e}_j) \rangle})^{v_{\pi(j)}v_{\pi(k)}}
    & =
    \prod_{1\le j <k \le n, \ 1\le \pi(j)< \pi(k)\le n} ({\bf q}^{\langle  \pi({\bf e}_k)|\pi({\bf e}_j) \rangle})^{v_{\pi(j)}v_{\pi(k)}}\\
    & \xlongequal{{\rm Lem.}\ \ref{propeqq}(2)}
    \prod_{1\le j <k \le n,\ 1\le \pi(j)< \pi(k)\le n} (q_{\pi(j)\pi(k)})^{v_{\pi(j)}v_{\pi(k)}}\\
    & =
    \prod_{ 1\le \pi^{-1}(s)< \pi^{-1}(t)\le n, \ 1\le s <t \le n} (q_{st})^{v_sv_t}\\
    & =
    \prod_{1\le \pi(j) <\pi(k) \le n,\ 1\le j< k\le n} (q_{jk})^{v_{j}v_{k}}\\
    & =
    \prod_{1\le \pi(j) <\pi(k) \le n,\ 1\le j< k\le n} (q_{\pi(k)\pi(j)})^{v_{j}v_{k}}\\
    & =
    \prod_{1\le \pi(j) <\pi(k) \le n,\ 1\le j< k\le n} (\frac{1}{{\bf q}^{\langle  \pi({\bf e}_k)|\pi({\bf e}_j) \rangle}})^{v_{k}v_{j}}\\
    & =
    \prod_{1\le j<k \le n}(\frac{1}{{\bf q}^{\langle\pi({\bf e}_k)|\pi({\bf e}_j) \rangle}})^{v_{j}v_{k}}.
\end{align*}
This proved the equality (6.2) below.

\vskip5pt

\begin{lemma} \label{g0a-1} \ For $A =  A({\bf q,a})$, assume that
there is a compatible permutation $\pi\in S_n$ with $\pi^2 = \Id$ and $c_i\in\mathbb K, \ 1\le i\le n$,  such that
$$c_ic_{\pi(i)}=h_{{\bf e}_i}, \ 1\le i, j\le n; \ \ \ \ \ \ \ q_\pi\prod\limits_{1\le i\le n} c_i^{a_i-1}=1.$$ Let $g_{\bf a-1-v, \pi(v)}\in\mathbb K$ be defined as in $(6.2)$.
Then

\begin{align}\label{equationpig}
    g_{\bf a-1-\pi(v),v}= \frac{1}{{\bf q}^{\langle {\bf a-1-\pi(v)|\pi(v)} \rangle}}\prod_{1\le i\le n} c_{\pi(i)}^{v_i}\prod_{1\le j<k \le n}(\frac{1}{{\bf q}^{\langle\pi({\bf e}_k)|\pi({\bf e}_j) \rangle}})^{v_{j}v_{k}}
    , \ \forall \ {\bf v}\in V.
\end{align}
And one has \ $g_{\bf a-1,0}=1 = g_{\bf 0,a-1}$.
\end{lemma}
\begin{proof} By (6.2) it is clear that $g_{\bf a-1,0}=1$. By (6.2) one has
    \[
    g_{\bf 0,a-1} = \frac{1}{{\bf q}^{\langle {\bf 0|a-1} \rangle}}\prod_{1\le i\le n} c_{i}^{a_i-1}\prod_{1\le j <k \le n} ({\bf q}^{\langle  \pi({\bf e}_k)|\pi({\bf e}_j) \rangle})^{(a_j-1)(a_k-1)}.
    \]
  By assumption $$\prod\limits_{1\le i\le n} c_{i}^{a_i-1}
    \prod\limits_{1\le j <k \le n} ({\bf q}^{\langle  \pi({\bf e}_k)|\pi({\bf e}_j) \rangle})^{(a_j-1)(a_k-1)} = q_\pi\prod\limits_{1\le i\le n} c_{i}^{a_i-1}=1$$ one gets $g_{\bf 0,a-1}=1$.
\end{proof}

\subsection{\bf Coalgebra structures on $A=  A({\bf q,a})$}

Define \ $\varepsilon: A\longrightarrow \mathbb{K}$ and $\Delta: A \longrightarrow A\otimes A$ to be the $\mathbb{K}$-linear maps as follows

\begin{equation}\label{eqcoalgonA}
\left\{
\begin{array}{l}
\varepsilon(x_\mathbf{v})=\delta_{\mathbf{v,0}}\ \ \text{for} \ \  \mathbf{v}\in V;\\

\Delta(1_A)= 1_A\otimes 1_A; \\

\Delta(x_\mathbf{v})= 1_A\otimes x_\mathbf{v} +x_\mathbf{v}\otimes 1_A \ \ \text{for} \ \  \mathbf{v}\in V-\{\bf 0, a-1\}; \\

\Delta(x_\mathbf{a-1}) = \sum\limits_{{\bf v}\in V}  g_{{\bf a-1-v},\pi({\bf v})} \ x_{\bf a-1-v}\otimes x_{\pi(\mathbf{v})}\\
\ \ \ \ \ \ \ \ \ \ \ \ = 1_A\otimes x_\mathbf{a-1} +x_\mathbf{a-1}\otimes 1_A + \sum\limits_{{\bf v}\in V-\{\bf 0, \ a-1\}}  g_{{\bf a-1-v},\pi({\bf v})} \ x_{\bf a-1-v}\otimes x_{\pi(\mathbf{v})}.$$
\end{array}
\right.
\end{equation}
where $g_{\bf a-1-v, \pi(v)}$ are given by (\ref{equationg}), see also Lemma \ref{g0a-1}.

\vskip5pt

Recall that for a coalgebra structure on $A$ with comultiplication \ $\Delta$, an element \ $x\in A$ is {\it primitive},  if \ $\Delta(x)=1_A\otimes x+ x\otimes 1_A$.
All the primitive elements of $A$ form a $\mathbb K$-linear space,  denoted by \ $P(A)$.

\begin{lemma}\label{lemcoalg}  The algebra $A =  A({\bf q,a})$ admits a coalgebra structure,
with $\varepsilon$ and $\Delta$ as in $(\ref{eqcoalgonA});$  and \ ${\rm dim} P(A)\ge (\prod\limits_{1\le i\le n}a_i) -2.$
\end{lemma}
\begin{proof} One directly verifies the coassociativity  \ $(\Delta \otimes {\rm Id})\Delta = ({\rm Id} \otimes \Delta)\Delta$ and the counit property \ $(\varepsilon\otimes {\rm Id})\Delta\cong {\rm Id} \cong ({\rm Id}\otimes\varepsilon)\Delta$. For example,
    \begin{align*}
    (& \Delta \otimes {\rm Id})\Delta (x_{\bf a-1})
    = 1_A\otimes 1_A\otimes x_{\bf a-1} + 1_A\otimes x_{\bf a-1} \otimes 1_A + x_{\bf a-1}\otimes 1_A \otimes 1_A \\
    &+\sum_{{\bf u, v}\in V -\{\bf 0, \ a-1\}} g_{{\bf a-1-v},\pi({\bf v})}
    (x_{\bf a-1-v}\otimes x_{\pi({\bf v})}\otimes 1_A + 1_A\otimes x_{\bf a-1-v}\otimes x_{\pi({\bf v})} + x_{\bf a-1-v}\otimes 1_A\otimes x_{\pi({\bf v})})\\
    =& ({\rm Id}\otimes \Delta)\Delta (x_{\bf a-1}).
    \end{align*}

By (6.3), all \ $x_{\bf v}$ with \ ${\bf v}\in V-\{\bf 0, a-1\}$ are $\mathbb K$-linear independent primitive elements of \ $A$. Hence
\ dim $P(A) \ge (\prod\limits_{1\le i\le n}a_i) -2.$ \end{proof}

\vskip5pt

\begin{rmk}\label{difference} The comultiplication \ $\Delta$ in {\rm (6.3)} is different from
the one given in {\rm [AS2, Example 1.9]:} there ${\rm dim} P(A)= n$.  Also, it generalizes one in {\rm [JZ, (3.1)]} where $\pi = \Id$.
\end{rmk}

\subsection{\bf Bi-Frobenius structures on $A=  A({\bf q,a})$} Fix the setting up in Subsection 6.1.
Choosing $\phi = x_{\bf a-1}^*$ and $t = x_{\bf a-1}$. Consider the $\mathbb K$-linear map $S: A\longrightarrow A$ given by

\begin{align}S(x_{\bf v}) & =\sum\phi(t_1x_{\bf v})t_2= g_{\bf a-1-v, \pi({\bf v})}{\bf q}^{\langle {\bf a-1-v|v}\rangle}x_{\pi({\bf v})} \nonumber
\\ \label{equationS}&= \prod_{1\le i\le n} c_{i}^{v_i}\prod_{1\le j <k \le n} ({\bf q}^{\langle  \pi({\bf e}_k)|\pi({\bf e}_j) \rangle})^{v_jv_k}x_{\bf \pi(v)}, \ \forall \ {\bf v}\in V.
\end{align}
In particular, \  $S(1_A) = 1_A;  \ \ S(x_{{\bf a-1}}) = g_{\bf 0, a-1}x_{{\bf a-1}} = x_{{\bf a-1}}$, and
$$S(x_i) = g_{{\bf a-1}-{\bf e}_i, \pi({\bf e}_i)}{\bf q}^{\langle {\bf a-1}-{\bf e}_i|{\bf e}_i \rangle}x_{\pi(i)}= c_ix_{\pi(i)}, \ 1\le i\le n.$$

\begin{theorem}\label{sufficiency} \ For $A =  A({\bf q,a})$, assume that
there is a compatible permutation $\pi\in S_n$ with $\pi^2 = \Id$, and that there are $c_i\in\mathbb K \ (1\le i\le n)$  such that
$$c_ic_{\pi(i)}=h_{{\bf e}_i},  \ \ \ \ \ \ \ q_\pi \prod\limits_{1\le i\le n} c_i^{a_i-1} =1$$
where $q_\pi = \prod\limits_{1\le j <k \le n} ({\bf q}^{\langle  \pi({\bf e}_k)|\pi({\bf e}_j) \rangle})^{(a_k-1)(a_j-1)}.$
Then   $(A, \ \phi= x_{\bf a-1}^*, \ t = x_{\bf a-1}, \ S)$ is a bi-Frobenius algebra,
where  the coalgebra structure is given by $(6.3)$ and $S$ is given by $(\ref{equationS})$. In particular, $S(x_i) = c_ix_{\pi(i)}, \ 1\le i\le n$.
\end{theorem}
\begin{proof} By Lemma \ref{lemcoalg}, $\varepsilon$ and $\Delta$ in (\ref{eqcoalgonA}) give a coalgebra structure on $A$. It remains to  check that the conditions (i), (ii) and (iii) in Definition \ref{definitionbf}
are satisfied. By definition $\varepsilon: A\longrightarrow \mathbb{K}$ is an algebra map, and  $1_A$ is a group-like element. By Lemma \ref{frobeniushomomorphism}, $(A, \ \phi)$ is a Frobenius algebra.

\vskip5pt

By (6.3) one has
\begin{align*}
    t\leftharpoonup x_{\bf a-1-v}^* & = \sum x_{\bf a-1-v}^*(t_1)t_2
    = \begin{cases} 1, \ & \mbox{if} \ \bf v = \bf 0; \\ x_{\bf a-1}, \ & \mbox{if} \ \bf v = {\bf a-1}; \\  g_{\bf a-1-v,\pi(v)}x_{\pi({\bf v})}, \ & \mbox{if} \ {\bf v}\in V-\{\bf 0, a-1\}\end{cases}.
\\ & = g_{\bf a-1-v, \pi(v)}x_{\pi({\bf v})}, \ \forall \ {\bf v}\in V.\end{align*}
    Since $\pi$ induces a permutation of $V$ and $g_{\bf a-1-v, \pi(v)}\ne 0$ for  ${\bf v}\in V$, one has $A = t\leftharpoonup A^*$, i.e., $(A, \ x_{\bf a-1})$ is a Frobenius coalgebra.

\vskip5pt

To show $S$ is an algebra anti-homomorphism of $A$, it suffices to show \ $S(x_{\bf u}x_{\bf v}) = S(x_{\bf v})S(x_{\bf u})$ for \ ${\bf u}, {\bf v}\in V$. By (\ref{equationS}) one has
\begin{align*}& S(x_\mathbf{v}) = \prod_{1\le i\le n}c_{i}^{v_i}
    \prod_{1\le j <k \le n} ({\bf q}^{\langle  \pi({\bf e}_k)|\pi({\bf e}_j) \rangle})^{v_jv_k}x_{\bf \pi(v)};\\
    & S(x_\mathbf{v})S(x_\mathbf{u})
    = {\bf q}^{\langle {\bf \pi(v)|\pi(u)}\rangle}
    \prod_{1\le i\le n} c_{i}^{u_i+v_i}
    \prod_{1\le j <k \le n} ({\bf q}^{\langle  \pi({\bf e}_k)|\pi({\bf e}_j) \rangle})^{v_jv_k+u_ju_k} x_{\pi({\bf u+v})};
    \\  & S(x_\mathbf{u}x_\mathbf{v})= {\bf q}^{\langle \mathbf{u|v}\rangle}S(x_{\bf u+v})\\
    & \ \ \ \ \ \ \ \ \ \ \ = {\bf q}^{\langle \mathbf{u|v}\rangle}
    \prod_{1\le i\le n} c_{i}^{u_i+v_i}
    \prod_{1\le j <k \le n} ({\bf q}^{\langle  \pi({\bf e}_k)|\pi({\bf e}_j) \rangle})^{(u_j+v_j)(u_k+v_k)} x_{\pi({\bf u + v})}. \end{align*}
By Lemma \ref{propusefuleq2}(3) one has
\ ${\bf q}^{\langle {\bf \pi(v)|\pi(u)}\rangle} ={\bf q}^{\langle{\bf u|v}\rangle} \prod\limits_{1\le j<k \le n} ({\bf q}^{\langle  \pi({\bf e}_k)|\pi({\bf e}_j) \rangle})^{u_jv_k+v_ju_k}.$
It follows that

    \begin{align*} S(x_\mathbf{u}x_\mathbf{v}) & = {\bf q}^{\langle \mathbf{u|v}\rangle}
    \prod_{1\le i\le n} c_{i}^{u_i+v_i}
    \prod_{1\le j <k \le n} ({\bf q}^{\langle  \pi({\bf e}_k)|\pi({\bf e}_j) \rangle})^{(u_j+v_j)(u_k+v_k)} x_{\pi({\bf u + v})}
    \\ & = {\bf q}^{\langle \mathbf{u|v}\rangle}
    \prod_{1\le i\le n} c_{i}^{u_i+v_i}
    \prod_{1\le j <k \le n} ({\bf q}^{\langle  \pi({\bf e}_k)|\pi({\bf e}_j) \rangle})^{v_jv_k +u_ju_k} ({\bf q}^{\langle  \pi({\bf e}_k)|\pi({\bf e}_j) \rangle})^{u_jv_k+v_ju_k} x_{\pi({\bf u + v})}
   \\ & = \frac{{\bf q}^{\langle \mathbf{u|v}\rangle}{\bf q}^{\langle{\bf \pi(v)|\pi(u)}\rangle}}{{\bf q}^{\langle{\bf u|v}\rangle}}
  \prod_{1\le i\le n} c_{i}^{u_i+v_i} \prod_{1\le j <k \le n} ({\bf q}^{\langle  \pi({\bf e}_k)|\pi({\bf e}_j) \rangle})^{v_jv_k +u_ju_k} x_{\pi({\bf u + v})}
      \\ &  = {\bf q}^{\langle {\bf \pi(v)|\pi(u)}\rangle}
    \prod_{1\le i\le n} c_{i}^{u_i+v_i}
    \prod_{1\le j <k \le n} ({\bf q}^{\langle  \pi({\bf e}_k)|\pi({\bf e}_j) \rangle})^{v_jv_k+u_ju_k} x_{\pi({\bf u+v})}
    \\ & = S(x_\mathbf{v})S(x_\mathbf{u}).
    \end{align*}

Now we verify that $S$ is a coalgebra anti-homomorphism of $A$, i.e., $\varepsilon\circ S=\varepsilon$ and $(S\otimes S)\circ\tau\circ\Delta = \Delta\circ S.$
Since $|x_{\pi({\bf v})}| = |x_{\bf v}|, \ \forall \ {\bf v}\in V$, one has
$$\varepsilon\circ S(x_{\bf v})= \delta_{\mathbf{v,0}} = \varepsilon(x_{\bf v}), \ \forall \ {\bf v}\in V$$
i.e., $\varepsilon\circ S=\varepsilon$. Also, one has $(S\otimes S)\circ\tau\circ\Delta(1)= 1\otimes 1 = (\Delta\circ S)(1).$ For ${\bf v}\in V- \{\bf 0,a-1\}$,  one has
$\pi({\bf v})\in V- \{\bf 0,a-1\}$ and
\begin{align*} &(S\otimes S)\circ\tau\circ\Delta(x_{\bf v})= 1\otimes S(x_{\bf v})+S(x_{\bf v})\otimes 1\\
= &   g_{\bf a-1-v, \pi({\bf v})}{\bf q}^{\langle {\bf a-1-v|v}\rangle}(1\otimes x_{\bf \pi(v)} + x_{\bf \pi(v)}\otimes 1) =g_{\bf a-1-v, \pi({\bf v})}{\bf q}^{\langle {\bf a-1-v|v}\rangle}\Delta(x_{\bf\pi(v)}) \\
    =&   (\Delta\circ S)(x_{\bf v}).
\end{align*}
Note that $(\Delta\circ S)(x_{\bf a-1}) = \Delta(x_{\bf a-1})= \sum\limits_{{\bf v}\in V}  g_{{\bf a-1-v},\pi({\bf v})} \ x_{\bf a-1-v}\otimes x_{\pi(\mathbf{v})}.$ Since $\pi^2 = \Id$, one has
\begin{align*}
    &(S\otimes S)\circ\tau\circ\Delta(x_{\bf a-1})= \sum_{{\bf v}\in V }  g_{\bf a-1-v,\pi(v)}S(x_{\bf \pi(v)})\otimes S(x_{\bf a-1-v})\\
    =& \sum_{{\bf v}\in V } g_{\bf a-1-v,\pi(v)}g_{\bf a-1-\pi(v), v}g_{\bf v,a-1-\pi(v)} {\bf q}^{\langle\bf a-1-\pi(v)|\pi(v)\rangle}{\bf q}^{\langle\bf v|a-1-v\rangle} x_{\bf v}\otimes x_{\bf a-1-\pi(v)}.
\end{align*}
For each ${\bf v}\in V$, comparing the coefficients of $x_{\bf v}\otimes x_{\bf a-1-\pi(v)}$, it suffices  to verify
$$g_{\bf a-1-v,\pi(v)}g_{\bf a-1-\pi(v), v}g_{\bf v,a-1-\pi(v)} {\bf q}^{\langle\bf a-1-\pi(v)|\pi(v)\rangle}{\bf q}^{\langle\bf v|a-1-v\rangle} = g_{{\bf v}, {\bf a-1-\pi(v)}}, \ \forall \ {\bf v}\in V$$
or, equivalently \[
    g_{\bf a-1-v,\pi(v)}g_{\bf a-1-\pi(v), v}{\bf q}^{\langle\bf a-1-\pi(v)|\pi(v)\rangle}{\bf q}^{\langle\bf v|a-1-v\rangle}=1, \ \forall  \ {\bf v}\in V.
\]
By construction (\ref{equationg}) and (\ref{equationpig}), one has
\begin{align*}
    g_{\bf a-1-v,\pi(v)} =\frac{1}{{\bf q}^{\langle {\bf a-1-v|v} \rangle}}\prod_{1\le i\le n} c_{i}^{v_i}\prod_{1\le j <k \le n} ({\bf q}^{\langle  \pi({\bf e}_k)|\pi({\bf e}_j) \rangle})^{v_jv_k}, \ \forall \ {\bf v}\in V
\end{align*}
and $$g_{\bf a-1-\pi(v), v} = \frac{1}{{\bf q}^{\langle {\bf a-1-\pi(v)|\pi(v)} \rangle}}\prod_{1\le i\le n} c_{\pi(i)}^{v_i}\prod_{1\le j<k \le n}(\frac{1}{{\bf q}^{\langle\pi({\bf e}_k)|\pi({\bf e}_j) \rangle}})^{v_{j}v_{k}}, \ \forall \ {\bf v}\in V.$$
Thus, by definition $h_{\bf v} =\frac{{\bf q}^{\bf \langle a-1-v|v\rangle}}{{\bf q}^{\bf \langle v|a-1-v\rangle}},  \ \forall \ {\bf v}\in V,$ one has
\begin{align*}
   & g_{\bf a-1-v,\pi(v)}g_{\bf a-1-\pi(v), v}{\bf q}^{\langle\bf a-1-\pi(v)|\pi(v)\rangle}{\bf q}^{\langle\bf v|a-1-v\rangle}
   \\ &=
    \frac{{\bf q}^{\langle\bf v|a-1-v\rangle}}
    {{\bf q}^{\langle {\bf a-1-v|v} \rangle}}
    \prod_{1\le i\le n}(c_ic_{\pi(i)})^{v_i} =  \frac{1}{h_{\bf v}}\prod_{1\le i\le n}(c_ic_{\pi(i)})^{v_i}, \ \forall \ {\bf v}\in V.\end{align*}
By Lemma \ref{propeqq}(5), $h_{\bf v} = \prod\limits_{1\le i\le n}{h_{{\bf e}_i}}^{v_i}$ for ${\bf v}\in V$. By assumption $c_ic_{\pi(i)}=h_{{\bf e}_i}, \ 1\le i\le n$. Hence
$$
g_{\bf a-1-v,\pi(v)}g_{\bf a-1-\pi(v), v}{\bf q}^{\langle\bf a-1-\pi(v)|\pi(v)\rangle}{\bf q}^{\langle\bf v|a-1-v\rangle} =
\prod\limits_{1\le i\le n}(\frac{c_ic_{\pi(i)}} {h_{{\bf e}_i}})^{v_i} =1, \ \forall \ {\bf v}\in V.
$$
This completes the proof of $(S\otimes S)\circ\tau\circ\Delta (x_{\bf a-1}) = (\Delta\circ S) (x_{\bf a-1})$, and hence $S$ is a coalgebra anti-homomorphism of $A$.
\vskip5pt

We have proved that  $(A, \ \phi= x_{\bf a-1}^*, \ t = x_{\bf a-1}, \ S)$ is a bi-Frobenius algebra structure  with coalgebra structure as in $(6.3)$ and $S$ as in $(\ref{equationS})$.
By the construction  in (\ref{equationS}), $S$ is a permutation antipode. This completes the proof. \end{proof}

\noindent {\bf Proof of Theorem \ref{mainthm}:} Now,  Theorem \ref{mainthm} follows from  Theorems \ref{thmperS} and \ref{sufficiency}.   \hfill $\Box$

\subsection{\bf Symmetric case}

\begin{corollary} \label{symmetriccase} \ Let $A=A({\bf q,a})$ be a symmetric algebra. Then $A$ admits a bi-Frobenius algebra structure with permutation antipode
if and only if there is a compatible permutation $\pi\in S_n$ with $\pi^2 = \Id$.
\vskip5pt

If this is the case, then  $(A, \ x_{\bf a-1}^*, \ x_{\bf a-1}, \ S)$ is a bi-Frobenius algebra,
where  the comultiplication $\Delta$ is given by $(6.3)$,  $S$ is given by $(\ref{equationS})$, and $S(x_i) = c_ix_{\pi(i)}$ for some $c_i\in\mathbb K, \ 1\le i\le n$. \end{corollary}

\begin{proof} \ By Theorem \ref{mainthm}, it suffices  to show the existence of $\{c_i\}_{1\le i\le n}\subset\mathbb{K}-\{0\}$ satisfying the conditions
$$c_ic_{\pi(i)}=h_{{\bf e}_i}, \ 1\le i\le n; \ \ \ \ \ \ q_\pi\prod\limits_{1\le i\le n} c_i^{a_i-1} =1.$$
Since $A$ is a symmetric algebra, each $h_{{\bf e}_i} = 1$ (cf. Lemma \ref{frobeniushomomorphism}).  By Lemma \ref{lemmaI}(4), these conditions can be rewritten as
   \[
    c_ic_{\pi(i)}=1, \ \forall \ 1\le i\le n; \ \ \ \ \
    \prod\limits_{1\le i\le n} c_i^{a_i-1}\prod_{\substack{1\le j <k \le n,\\ j,\ k\  \in\ I.}}  (q_{kj})^{(a_k-1)(a_j-1)}=1
   \]
where $I$ is the same as in Lemma \ref{lemmaI}, i.e., $I=\{i \ | \ 1\le i\le n,  \ \pi(i)=i \}$. Also, recall that $J =\{i \ | \ 1\le i\le n, \ \pi(i)\ne i \}$.
We choose $c_i$ as follows:
$$c_i = \begin{cases}\prod\limits_{j\ge i, j \in I}  (q_{ij})^{(a_j-1)}, &  \mbox{if} \ i\in I;\\ 1, & \mbox{if} \ i\in J. \end{cases}$$

If $i\in J$, then $\pi(i)\in J$, and hence $c_ic_{\pi(i)}=1$ for $i\in J$.

\vskip5pt

If $i\in I$, then $\pi(i) = i$ and
$$c_ic_{\pi(i)} = c_i^2 = \prod\limits_{j \ge i,\ j \in I}  (q_{ij})^{2(a_j-1)}=\prod\limits_{j \ge i,\ j \in I}  (q_{ij}q_{\pi(i)\pi(j)})^{(a_j-1)} = \prod\limits_{j \ge i,\ j \in I}  (q_{ij}q_{ji})^{(a_j-1)} = 1.$$

\vskip5pt

On the other hand, one has
    \begin{align*} &\prod\limits_{1\le i\le n} c_i^{a_i-1}\prod_{\substack{1\le j <k \le n,\\ j,\ k\  \in\ I}}  q_{kj}^{(a_k-1)(a_j-1)}
    = \prod\limits_{i\in I} c_i^{a_i-1}\prod_{\substack{1\le j <k \le n,\\ j,\ k\  \in\ I}}  q_{kj}^{(a_k-1)(a_j-1)}
    \\& = \prod\limits_{i\in I}\prod\limits_{j\ge i, j \in I}  q_{ij}^{(a_i-1)(a_j-1)} \prod_{\substack{1\le j <k \le n,\\ j,\ k\  \in\ I}}  q_{kj}^{(a_k-1)(a_j-1)}
    \\& = \prod_{\substack{1\le k\le  j \le n\\ j,\ k\  \in\ I}}  q_{kj}^{(a_k-1)(a_j-1)}
        \prod_{\substack{1\le j <k \le n\\ j,\ k\  \in\ I}}  q_{kj}^{(a_k-1)(a_j-1)}
     =    \prod_{ j, k \in I} q_{kj}^{(a_k-1)(a_j-1)}\\
        = &
        \prod_{ k \in I}\prod_{ j \in I} q_{kj}^{(a_k-1)(a_j-1)}
        \xlongequal{{\rm Lem.}\ \ref{lemmaI}(3)}
        \prod_{ k \in I} h_{{\bf e}_k}^{a_k-1}=1.
    \end{align*}
This completes the proof.
\end{proof}

\subsection{\bf The case of ${\rm Char}\mathbb K = 2$}

\begin{corollary} \label{char2} \  Assume that ${\rm Char}\mathbb{K}=2$. Then $A = A({\bf q,a })$ admits a bi-Frobenius algebra structure with permutation antipode
if and only if $\mathcal{ N}^2=\Id$ and there is  compatible permutation $\pi\in S_n$ with  $\pi^2 = \Id$.

\vskip5pt

If this is the case, then $A$ is a symmetric algebra and $(A, \ x_{\bf a-1}^*, \ x_{\bf a-1}, \ S)$ is a bi-Frobenius algebra with  permutation antipode,
where the comultiplication $\Delta$ is given by $(6.3)$, $S$ is given by $(\ref{equationS})$, and $S(x_i) = x_{\pi(i)}, \ 1\le i\le n$.
 \end{corollary}
\begin{proof} \ If $A$ admits a bi-Frobenius algebra with  permutation antipode, then $\mathcal{N}^2=\Id$ by Proposition \ref{thenecessity}(iv); and
there is a compatible permutation $\pi\in S_n$ with  $\pi^2 = \Id$, by Theorem \ref{mainthm}.

\vskip5pt

Conversely, assume that  $\mathcal{ N}^2=\Id$ and there is a compatible  permutation $\pi\in S_n$ with  $\pi^2 = \Id$.
Then $h_{{\bf e}_i}^2=1, \ 1\le i\le n$ (cf. Lemma \ref{frobeniushomomorphism}). Since ${\rm Char}\mathbb{K}=2$, one gets
$h_{{\bf e}_i}=1, \ 1\le i\le n$. Thus $A$ is a symmetric algebra.

\vskip5pt

By Lemma \ref{lemmaI}(4), $q_\pi =\prod\limits_{1\le j <k \le n} ({\bf q}^{\langle  \pi({\bf e}_k)|\pi({\bf e}_j) \rangle})^{(a_k-1)(a_j-1)}= \pm 1 \xlongequal{{\rm Char}\mathbb{K}=2}1$.
Take $c_i = 1$,  $1\le i\le n$. Then
    $$c_ic_{\pi(i)}=1= h_{{\bf e}_i}, \ 1\le i\le n, \ \ \ \ \
    q_\pi \prod\limits_{1\le i\le n} c_i^{a_i-1}=1.$$
    Thus, by Theorem \ref{mainthm},  $(A, \ x_{\bf a-1}^*, \ x_{\bf a-1}, \ S)$ is a bi-Frobenius algebra with  permutation antipode, where the comultiplication $\Delta$ is given by $(6.3)$, $S$ is given by $(\ref{equationS})$, and $S(x_i) = x_{\pi(i)}, \ 1\le i\le n$.
    This completes the proof.
\end{proof}

\subsection{Examples} We show by examples that the conditions in Theorem \ref{mainthm} (or Theorem \ref{sufficiency}) can be realized for non-commutative algebras (i.e., not all $q_{ij}$ are $1$).
\begin{examples} \ Let $A= A({\bf q, a})$ with ${\bf a}= (2,2,2)$ and ${\bf q} = \left(\begin{smallmatrix}1 &b &\frac{1}{b}\\ \frac{1}{b} &1&b\\ b&\frac{1}{b}&1 \end{smallmatrix}\right) $, where \ $0\ne b\in \mathbb{K}.$
Thus $A = \mathbb{K}\langle x_1, x_2, x_3 \rangle/\langle x_1^{2}, \ x_2^{2}, \ x_3^{2}, \ x_2x_1-bx_1x_2, \ x_3x_1-\frac{1}{b}x_1x_3, \ x_3x_2- bx_2x_3\rangle.$ Then
$$V=\{(0,0,0), \ (1, 0, 0), \ (0,1,0), \ (0,0,1), \ (1,1,0), \ (1,0,1), \ (0,1,1), \ (1,1,1)\}$$
and $h_{{\bf e}_1}=h_{{\bf e}_2}=h_{{\bf e}_3}=1$. Thus $A$ is symmetric $(${\rm cf. Lemma \ref{frobeniushomomorphism}} $).$ Taking $\pi=(2,3)\in S_3$,  then  $q_{\pi(i)\pi(j)}=q_{ji}$ for $1\le i, j\le 3$.
By {\rm Theorem \ref{symmetriccase}},  $A$ has a bi-Frobenius algebra structure with permutation antipode.

\vskip5pt

In fact, as in the proof of {\rm Theorem \ref{symmetriccase}},
we take $c_1 = c_2 = c_3 =1$. Consider  the $\mathbb{K}$-linear maps \ $\varepsilon: A\longrightarrow \mathbb{K}$, \  $\Delta: A \longrightarrow A\otimes A$ and
\ $S: A\rightarrow A$, as in {\rm (6.3)} and {\rm (6.4):}
\begin{equation*}
\left\{\begin{aligned}&\varepsilon(x_\mathbf{v})=\delta_{\mathbf{v,0}}\ \ \text{for all} \ \  \mathbf{v}\in V; \\ & \Delta(1)= 1\otimes 1;
\\ &\Delta(x_\mathbf{v})= 1\otimes x_\mathbf{v} +x_\mathbf{v}\otimes 1, \ \ \forall \ \  \mathbf{v}\in V-\{(0,0,0),(1,1,1)\}, \\
&\Delta(x_1x_2x_3) = 1\otimes x_{1}x_2x_3 + x_{1}x_2x_3\otimes 1\\
&\hskip47pt + x_2x_3\otimes x_1 + \frac{1}{b}x_1x_3\otimes x_3 + x_1x_2\otimes x_2\\
&\hskip47pt + \frac{1}{b} x_3\otimes x_1x_3 + x_2\otimes x_1x_2 + x_1\otimes x_2x_3;\\
& S(1)=1; \quad S(x_1) = x_1; \quad S(x_2) = x_3; \quad S(x_3) = x_2;\\
&S(x_1x_2)=\frac{1}{b}x_1x_3; \quad S(x_1x_3) = bx_1x_2; \quad S(x_2x_3) = x_2x_3;\\
&S(x_1x_2x_3)=x_1x_2x_3.
\end{aligned}\right.\end{equation*}
Then $(A, \ (x_1x_2x_3)^{*}, \ x_1x_2x_3, \ S)$ is a bi-Frobenius algebra with permutation antipode $($one can also verify this by definition$)$.
\end{examples}

\begin{examples} \label{exmnonsymmtric} \ Let $A= A({\bf q, a})$ with ${\bf a}= (2,2,2)$ and ${\bf q} = \left(\begin{smallmatrix} 1 &b &\frac{1}{b}\\ \frac{1}{b} &1&-b\\ b&-\frac{1}{b} &1 \end{smallmatrix}\right)$, where \ $0\ne b\in \mathbb{K}$,
    i.e., $A = \mathbb{K}\langle x_1, \ x_2, \ x_3\rangle/\langle x_1^{2}, \ x_2^{2}, \ x_3^{2}, \ x_2x_1-bx_1x_2, \ x_3x_1-\frac{1}{b}x_1x_3, \ x_3x_2+bx_2x_3\rangle$.
    Then \ $V=\{(0,0,0),(1,0,0),(0,1,0),(0,0,1),(1,1,0),(1,0,1),(0,1,1),(1,1,1)\}$ and $h_{{\bf e}_1}=1$, $h_{{\bf e}_2}=h_{{\bf e}_3}=-1$. Thus $A$ is not symmetric $($ {\rm cf. Lemma \ref{frobeniushomomorphism}} $).$ Taking $\pi=(2,3)\in S_3$, $c_1=-1$, $c_2=c_3=\sqrt{-1}$. Then  $q_{\pi(i)\pi(j)}=q_{ji}$ for $1\le i,j\le 3,$ and one has
    $$c_1c_{\pi(1)}=c_1^2=h_{{\bf e}_1}=1, \ \ \ c_2c_{\pi(2)}=c_3c_{\pi(3)}= c_2c_3 =h_{{\bf e}_2}=h_{{\bf e}_3} =-1$$
    and
    \begin{align*}
      \prod_{1\le j <k \le n} ({\bf q}^{\langle  \pi({\bf e}_k)|\pi({\bf e}_j) \rangle})^{(a_k-1)(a_j-1)}
      &=
       {\bf q}^{\langle  \pi({\bf e}_2)|\pi({\bf e}_1) \rangle}{\bf q}^{\langle  \pi({\bf e}_3)|\pi({\bf e}_1) \rangle}{\bf q}^{\langle  \pi({\bf e}_3)|\pi({\bf e}_2) \rangle} \\
       &=
       {\bf q}^{\langle  {\bf e}_3|{\bf e}_1 \rangle}{\bf q}^{\langle  {\bf e}_2|{\bf e}_1 \rangle}{\bf q}^{\langle  {\bf e}_2|{\bf e}_3 \rangle}=\frac{1}{b}\cdot b\cdot 1 = 1.
    \end{align*}
    Thus all the conditions in {\rm Theorem \ref{mainthm}} are satisfied. Consider $\mathbb{K}$-linear maps $\varepsilon: A\rightarrow \mathbb{K}$, $\Delta: A\rightarrow A\otimes A$ and $S:A\rightarrow A$
    as in {\rm (6.3)} and {\rm (6.4):}

    \begin{equation*}
    \left\{\begin{aligned}
    &\varepsilon(x_\mathbf{v})=\delta_{\mathbf{v,0}}\ \ \text{for all} \ \  \mathbf{v}\in V;\\
    &\Delta(1)= 1\otimes 1; \\
    &\Delta(x_\mathbf{v})= 1\otimes x_\mathbf{v} +x_\mathbf{v}\otimes 1, \ \ \forall \ \  \mathbf{v}\in V-\{(0,0,0),(1,1,1)\}; \\
    &\Delta(x_1x_2x_3) = 1\otimes x_{1}x_2x_3 + x_{1}x_2x_3\otimes 1\\
    &\hskip47pt - x_2x_3\otimes x_1 - \frac{\sqrt{-1}}{b}x_1x_3\otimes x_3 + \sqrt{-1} x_1x_2\otimes x_2\\
    &\hskip47pt + \frac{\sqrt{-1}}{b}x_3\otimes x_1x_3 - \sqrt{-1} x_2\otimes x_1x_2 - x_1\otimes x_2x_3;\\
    &S(1)=1; \quad S(x_1) = -x_1; \quad S(x_2) = \sqrt{-1}x_3; \quad S(x_3) = \sqrt{-1}x_2;\\
    &S(x_1x_2)=-\frac{\sqrt{-1}}{b}x_1x_3; \quad S(x_1x_3) = -b\sqrt{-1}x_1x_2; \quad S(x_2x_3) = -x_2x_3;\\
    &S(x_1x_2x_3)=x_1x_2x_3.
    \end{aligned}\right.\end{equation*}
    \vskip5pt \noindent By {\rm Theorem \ref{mainthm}}, $(A, \ (x_1x_2x_3)^{*}, \ x_1x_2x_3, \ S)$ is a bi-Frobenius algebra with permutation antipode $($one can also directly verify this by definition$)$.

    \vskip5pt

    Let $z=1+x_1\in U(A)$, and $\phi =z\rightharpoonup (x_1x_2x_3)^*$. Then $z^{-1}=1-x_1$, $\mathcal{N}(z)= \mathcal{N}(1) + \mathcal{N}(x_1) = 1 + h_{{\bf e}_1}x_1 = z$ and $\mathcal{N}(z^{-1})=z^{-1}$. By {\rm Lemma \ref{frobeniushomomorphism}} and $\mathcal N^2 = \Id$ one has
    \begin{align*}
        \mathcal{N}_{\phi}^2(x_2) & =\mathcal{N}_{\phi}(z\mathcal{N}(x_2)z^{-1}) = z\mathcal{N}(z\mathcal{N}(x_2)z^{-1})z^{-1}\\
        & = z\mathcal{N}(z)x_2\mathcal{N}(z^{-1})z^{-1} = z^2x_2z^{-2} = (1+2x_1)x_2(1-2x_1)\\
        & = x_2 + 2(1-b)x_1x_2.
    \end{align*}
    Thus, if Char $\mathbb K\ne 2$ and $b\ne 1$, then $\mathcal{N}_{\phi}^2\ne \Id$. This is in contrast to the fact $\mathcal N^2 = \Id.$
    \end{examples}

\section{\bf Intrinsic descriptions for  $A({\bf q,a})$ with permutation antipode}
Theorem \ref{mainthm} gives a sufficient and necessary condition for quantum complete intersection $A=A({\bf q,a})$ admitting a bi-Frobenius algebra structure with permutation antipode.
This condition involves the existence of $c_i\in\mathbb K$, satisfying some equalities. The aim of this section is to
replace the existence of $c_i$ by some intrinsic conditions which only involve the properties of $\bf q$ and $\bf a$.

\vskip5pt

Since the situation of Char $\mathbb K =2$ has been treated in Corollary \ref{char2}, we will assume Char $\mathbb K\ne 2$.

\subsection{Partitions of $\{1, \cdots, n\}$} Assume that $\mathcal N^2 = \Id$ (or equivalently, $h_{{\bf e}_i}^2=1, \ 1\le i\le n$), where $\mathcal N$ is the canonical Nakayama automorphism of $A$, and that
there is a compatible permutation $\pi\in S_n$ with $\pi^2 = \Id$. Thus $a_{\pi(i)}=a_i$ and $q_{\pi(i)\pi(j)}=q_{ji}$ for $1\le i,j\le n$.

\vskip5pt

Since $h_{{\bf e}_i}^2=1$ for $1\le i\le n$, we put

\begin{align*}I=\{i \ | \ 1\le i\le n,  \ \pi(i)=i \},  \ \ \   J =\{i \ | \ 1\le i\le n, \ \pi(i)\ne i \},\end{align*}
\begin{equation}\label{equationIJ}
\begin{aligned}I_1 &= \{i\in I \ |\ h_{{\bf e}_i}=1, \ a_i\  \text{ is \  even}\}, & J_1 &= \{i\in J \ |\ h_{{\bf e}_i}=1,\ a_i \ \text{ is \  even}\},\\
    I_2 &= \{i\in I \ |\ h_{{\bf e}_i}=1,\ a_i\  \text{ is \  odd}\}, & J_2 &= \{i\in J \ |\ h_{{\bf e}_i}=1,\ a_i \ \text{ is \  odd}\},\\
    I_3 &= \{i\in I \ |\ h_{{\bf e}_i}=-1,\ a_i\  \text{ is \  even}\}, &  J_3 &= \{i\in J \ |\ h_{{\bf e}_i}=-1,\ a_i \ \text{ is \  even}\},\\
    I_4 &= \{i\in I \ |\ h_{{\bf e}_i}=-1,\ a_i\  \text{ is \  odd}\}, & J_4 &= \{i\in J \ |\ h_{{\bf e}_i}=-1,\ a_i \ \text{ is \  odd}\}.
\end{aligned}
\end{equation}
Then $\{1, \cdots, n\} = I\cup J$. Since ${\rm Char}\mathbb{K} \ne 2$, one has also disjoint unions \ $I = \bigcup\limits_{1\le i\le 4} I_i$ and \ $J = \bigcup\limits_{1\le i\le 4} J_i.$ \  Let
$|I|$ denote the number of the elements in $I$.

\begin{lemma}\label{lemmaeven} Assume that ${\rm Char}\mathbb{K} \ne 2$. For $A=A({\bf q, a})$  with $h_{{\bf e}_i}^2=1$ for $1\le i\le n$, assume that there is a compatible permutation $\pi\in S_n$ with $\pi^2 = \Id$.  Then \vskip5pt
    {\rm (1)} \ \ $|I_3|$ is even.
    \vskip5pt
    {\rm (2)} \ \ If $|I_1|+|I_3|=0$, then $|I_4|=0$.
\end{lemma}
\begin{proof}
    {\rm (1)} \    Since $\prod\limits_{1\le i\le n}h_{{\bf e}_i}^{a_i-1}=h_{\bf a-1}=1$, and $h_{{\bf e}_i}^{a_i-1}=1$ for $ i\in \{1,\cdots, n\}- (I_3\cup J_3),$ one has
    \begin{align*}
        1=\prod_{1\le i\le n}h_{{\bf e}_i}^{a_i-1} & = \prod_{i\in I_3\cup J_3}h_{{\bf e}_i}^{a_i-1}.
    \end{align*}
If $i\in J_3$, then $\pi(i)\ne i$ and $\pi(i)\in J_3$. By Lemma \ref{propusefuleq2}(4), $h_{{\bf e}_i}h_{\pi({\bf e}_i)}=1$ for $1\le i\le n$. Thus
    \[\prod_{i\in J_3}h_{{\bf e}_i}^{a_i-1}\xlongequal {{\rm Lem. \ref{lemmaI}(1)}} \prod_{i\in J_3, i< \pi(i)}h_{{\bf e}_i}^{a_i-1}h_{\pi({\bf e}_i)}^{a_{\pi(i)}-1}= \prod_{i\in J_3, i< \pi(i)}(h_{{\bf e}_i}h_{\pi({\bf e}_i)})^{a_i-1}=1,
    \]
Therefore $\prod\limits_{i\in I_3}h_{{\bf e}_i}^{a_i-1}= 1$. However, for each $i\in I_3$, $h_{{\bf e}_i}^{a_i-1}=-1$, thus $|I_3|$ is even.
\vskip5pt
    {\rm (2)} \  Suppose that $|I_1|+|I_3|= 0$. By Lemma \ref{lemmaI}(3), one has
    $h_{{\bf e}_i} = \prod\limits_{j \in I}  (q_{ij})^{(a_j-1)}= \prod\limits_{j \in I_2\cup I_4}  (q_{ij})^{(a_j-1)}$, $\forall \ i\in I$. Since $q_{ij}^2 = 1$ for $i,j\in I$ (cf. Lemma \ref{lemmaI}(2))
    and $a_i-1$ is even for $i\in I_2\cup I_4$, one gets
    $
        h_{{\bf e}_i} =  \prod\limits_{j \in I_2\cup I_4}  (q_{ij})^{(a_j-1)} = 1, \ \forall \ i\in I.
    $
    However, by definition  $h_{{\bf e}_i}=-1$ for $i\in I_4$. It follows that $|I_4|=0$.
\end{proof}

\begin{lemma}\label{IJ} \ Assume that ${\rm Char}\mathbb{K} \ne 2$. Suppose that $A=A({\bf q,a})$ admits a bi-Frobenius algebra structure with permutation antipode.
Then there is a compatible permutation $\pi\in S_n$ with $\pi^2 = \Id$, such that $|I_1|+|I_3|\ne 0$ or that $\frac{|J_3|}{2}$ is even.
\end{lemma}
\begin{proof} By Theorem \ref{mainthm}, there is a compatible permutation $\pi\in S_n$ with $\pi^2 = \Id$ and $c_i\in\mathbb K, \ 1\le i\le n$, such that $c_ic_{\pi(i)}=h_{{\bf e}_i}$ for $1\le i\le n$
and that $$q_\pi\prod\limits_{1\le i\le n}c_i^{a_i-1} = 1\eqno(*)$$
where $q_\pi = \prod\limits_{j, k \in I, j <k} (q_{kj})^{(a_k-1)(a_j-1)}$ (cf.  Lemma \ref{lemmaI}(4)). \ Write $I_4$ as a disjoint union $I = I_{41}\cup I_{42}$, where
$$I_{41}= \{ i \in I_4 \ |\  \frac{a_i-1}{2}\ \ \text{is \  odd} \}, \ \ \ \ \ I_{42} = \{i \in I_4 \ |\  \frac{a_i-1}{2}\ \ \text{ is \  even}\}.$$
By the definition of $I_i$ and $J_i \ (1\le i\le 4)$ and $I_{4j} \ (j=1, 2)$ and by $c_ic_{\pi(i)}=h_{{\bf e}_i} \ (1\le i\le n)$, one has the following two tables:
\begin{table}[htbp]
	\centering
	\caption{When $i\in I$}
	\label{table1}
	\begin{tabular}{|c|c|c|c|c|c|}
		\hline
		& & & & & \\[-6pt]
		& $i\in I_1$ & $i\in I_2$ & $i\in I_3$ & $i\in I_{41}$ & $i\in I_{42}$ \\
        \hline
        & & & & &\\[-6pt]
        $c_i^2$ & 1 & 1 & -1 & -1 & -1  \\
        \hline
        & & & & &\\[-6pt]
        $c_i$ & $\pm 1$ & $\pm 1$ & $\pm\sqrt{-1} $ & $\pm\sqrt{-1} $ & $\pm\sqrt{-1}$ \\
		\hline
		& & & & &\\[-6pt]
		$c_i^{a_i-1}$ & $\pm 1$ & 1 & $\pm\sqrt{-1} $ & -1 & 1  \\
		\hline
	\end{tabular}
\end{table}
\begin{table}[htbp]
	\centering
	\caption{When $i\in J$}
	\label{table2}
	\begin{tabular}{|c|c|c|c|c|}
		\hline
		& & & &  \\[-6pt]
		&$i\in J_1$&$i\in J_2$&$i\in J_3$&$i\in J_{4}$ \\
        \hline
        & & & & \\[-6pt]
        $c_ic_{\pi(i)}$&1&1&-1&-1 \\
		\hline
		& & & & \\[-6pt]
		$(c_ic_{\pi(i)})^{a_i-1}$&1 &1&-1 &1  \\
		\hline
	\end{tabular}
\end{table}

Suppose that $|I_1|+|I_3|= 0$. Then $|I_4|=0$, by Lemma \ref{lemmaeven}(2).
Since
$I= \bigcup\limits_{1\le i\le 4} I_i= I_2$ and $a_i-1$ is even for $i\in I_2$,   one has
$$q_{\pi}=\prod\limits_{j, k \in I, j <k} (q_{kj})^{(a_k-1)(a_j-1)}= \prod\limits_{j, k \in I_2, j <k} (q_{kj})^{(a_k-1)(a_j-1)}=1.$$
Thus by (*) one gets $\prod\limits_{1\le i\le n} c_i^{a_i-1}=1.$

On the other hand, by Table (1), one has $\prod\limits_{i\in I_2} c_i^{a_i-1}=1$. Thus
\begin{align*}
  1=  \prod\limits_{1\le i\le n} c_i^{a_i-1} &= \prod\limits_{i\in I} c_i^{a_i-1}\prod\limits_{i\in J}c_i^{a_i-1} =\prod\limits_{i\in I_2} c_i^{a_i-1}\prod\limits_{i\in J} c_i^{a_i-1}
    \\
    &\xlongequal{{\rm Lem. \ref{lemmaI} (1)}} \prod\limits_{i\in J, i< \pi(i)} c_i^{a_i-1}\prod\limits_{i\in J, i< \pi(i)} (c_{\pi(i)})^{a_{\pi(i)}-1}
    \\
    & = \prod\limits_{i\in J, i< \pi(i)} (c_ic_{\pi(i)})^{a_i-1}
    \\
    & = \prod\limits_{i\in J_1\cup J_2\cup J_4, i< \pi(i)} (c_ic_{\pi(i)})^{a_i-1}\prod\limits_{i\in J_3, i< \pi(i)} (c_ic_{\pi(i)})^{a_i-1}
    \\ & \xlongequal{\rm Table \ (2)} \prod\limits_{i\in J_3, i< \pi(i)} (c_ic_{\pi(i)})^{a_i-1} = \prod\limits_{i\in J_3, i<\pi(i)} (-1)
    \\
    &= (-1)^{ \frac{|J_3|}{2}}.
\end{align*}
Since ${\rm Char} \mathbb{K} \ne 2$,   $\frac{|J_3|}{2}$ is even.
\end{proof}

\subsection{Main result I: the case of $\sqrt{-1}\in \mathbb{K}$}

\begin{theorem} \label{mainthm2} \ Assume that ${\rm Char}\mathbb{K}\ne 2$ and $\sqrt{-1}\in \mathbb{K}$. Then $A= A({\bf q,a })$ admits a bi-Frobenius algebra structure with permutation antipode
if and only if $\mathcal{N}^2=\Id$ and there is a compatible permutation $\pi\in S_n$ with $\pi^2 = \Id$, such that $|I_1|+|I_3|\ne 0$ or $\frac{|J_3|}{2}$ is even, where $I_1, I_3$ and $J_3$ are defined in $(\ref{equationIJ})$.
\vskip5pt

\vskip5pt

If this is the case, then  $(A, \ x_{\bf a-1}^*, \ x_{\bf a-1}, \ S)$ is a bi-Frobenius algebra,
where  the comultiplication $\Delta$ is given by $(6.3)$ and $S$ is given by $(\ref{equationS})$. \end{theorem}
\begin{proof} Assume that $A({\bf q,a })$ is a bi-Frobenius algebra  with permutation antipode. Then $\mathcal{N}^2=\Id$, by Proposition \ref{thenecessity}(iv).
By  Lemma \ref{IJ}, there is a compatible permutation $\pi\in S_n$ with $\pi^2 = \Id$ such that $|I_1|+|I_3|\ne 0$ or $\frac{|J_3|}{2}$ is even.

\vskip5pt

Conversely, assume that $\mathcal{N}^2=\Id$ and there is a compatible permutation $\pi\in S_n$ with $\pi^2 = \Id$ such that $|I_1|+|I_3|\ne 0$ or $\frac{|J_3|}{2}$ is even.
By Theorem \ref{mainthm}, it suffices  to show that there are $c_i\in\mathbb{K}, \ 1\le i\le n, $ such that
    \[
        c_ic_{\pi(i)}=h_{{\bf e}_i}, \ 1\le i\le n; \ \ \ \ q_\pi\prod\limits_{1\le i\le n} c_i^{a_i-1}=1
    \]
where \ $q_\pi = \prod\limits_{j,k\in I, j <k} ({\bf q}^{\langle  \pi({\bf e}_k)|\pi({\bf e}_j) \rangle})^{(a_k-1)(a_j-1)} = \pm 1$ (cf. Lemma \ref{lemmaI}(4)). By Lemma \ref{propusefuleq2}(4), $h_{{\bf e}_i} h_{{\bf e}_{\pi(i)}} = 1, \ 1\le i\le n$. Since $h_{{\bf e}_i}^2=1$ for $1\le i\le n$,  $h_{{\bf e}_{\pi(i)}}=h_{{\bf e}_i} = \pm 1$, $1\le i\le n$.
We divide the proof into three cases.

\vskip5pt

{\bf Case 1:  \ $|I_1|\ne 0$}. In this case we choose $i_0\in I_1$ and  $c_i\in\mathbb K \ (1\le i\le n)$ as follows:
$$c_i = \begin{cases}1, & \mbox{if} \ \ i\in I_1\cup I_2-\{i_0\};\\
    \sqrt{-1}, & \mbox{if} \ \ i\in I_3\cup I_4;\\
    1, & \mbox{if} \ \ i\in J, \ i<\pi(i);\\
    h_{{\bf e}_i}, & \mbox{if} \ \ i\in J, \ i>\pi(i);\\
    q_\pi\prod\limits_{i\in \{1,\cdots, n\}-\{i_0\}}c_i^{a_i-1}, &  \mbox{if} \ \ i=i_0.
\end{cases}$$
By the definition of $I_i, \ 1\le i\le 4$, one has
$$c_ic_{\pi(i)} = \begin{cases}c_i^2=1=h_{{\bf e}_i}, & \mbox{if} \ \ i\in I_1\cup I_2-\{i_0\};\\
    c_i^2=-1=h_{{\bf e}_i}, & \mbox{if} \ \ i\in I_3\cup I_4;\\
    1\cdot c_{\pi(i)}\xlongequal{} h_{{\bf e}_{\pi(i)}}=h_{{\bf e}_i}, & \mbox{if} \ \ i\in J, \ i<\pi(i);\\
    h_{{\bf e}_i}c_{\pi(i)}=h_{{\bf e}_i}\cdot 1 = h_{{\bf e}_i}, & \mbox{if} \ \ i\in J, \ i>\pi(i).
\end{cases}$$
This shows that $c_ic_{\pi(i)}=h_{{\bf e}_i}$ for $i\in \{1,\cdots, n\}-\{i_0\}$. By Lemma \ref{lemmaeven}(1), $|I_3|$ is even, and hence
$\prod\limits_{i\in I_3}c_i^{a_i-1}=\pm 1$ since $a_i-1$ is odd for $i\in I_3$.
Notice that $a_i-1$ is even for $i\in I_4$ and $h_{{\bf e}_i}=\pm 1$ for $1\le i\le n$.
Thus
\begin{align*}
    c_{i_0}&=q_\pi\prod\limits_{i\in \{1,\cdots, n\}-\{i_0\}}c_i^{a_i-1}\\
    &  \xlongequal{{\rm Lem.} \ref{lemmaI}(4)}  \pm \prod\limits_{i\in I_1\cup I_2 -\{i_0\}}c_i^{a_i-1}\prod\limits_{i\in I_3}c_i^{a_i-1}\prod\limits_{i\in I_4}c_i^{a_i-1}\prod\limits_{i\in J}c_i^{a_i-1}=  \pm 1.
\end{align*}
Therefore $c_{i_0}c_{\pi(i_0)}=c_{i_0}^2= 1 = h_{{\bf e}_{i_0}}$, where the last equality follows from the definition of $I_1$ and $i_0\in I_1$.

\vskip5pt
Moreover, notice that $a_{i_0}$ is even. Thus
\begin{align*} q_\pi \prod\limits_{1\le i\le n}c_i^{a_i-1}&  = q_\pi c_{i_0}^{a_{i_0}-1}\prod\limits_{i\in \{1,\cdots, n\}-\{i_0\}}c_i^{a_i-1} \\  & =  c_{i_0}^{a_{i_0}-1} (q_\pi\prod\limits_{i\in \{1,\cdots, n\}-\{i_0\}}c_i^{a_i-1}) \\  & = c_{i_0}^{a_{i_0}-1} c_{i_0} = c_{i_0}^{a_{i_0}} = (\pm 1)^{a_{i_0}} =1.
\end{align*}
Now, by Theorem \ref{mainthm},  $A$ admits a bi-Frobenius algebra structure with permutation antipode.

\vskip5pt

{\bf Case 2: \ $|I_3|\ne 0$}. \  In this case we choose $i_0\in I_3$ and $c_i\in\mathbb K \ (1\le i\le n)$ as follows:
$$c_i = \begin{cases}
    1, & \mbox{if} \ i\in I_1\cup I_2;\\
    \sqrt{-1}, & \mbox{if} \ i\in I_3\cup I_4-\{i_0\};\\
    1, & \mbox{if} \ i\in J, \ i<\pi(i);\\
    h_{{\bf e}_i}, & \mbox{if} \ i\in J, \ i>\pi(i);\\
    (-1)^{\frac{a_{i_0}}{2}}q_{\pi}\prod\limits_{i\in \{1, \cdots, n\}-\{i_0\}}c_i^{a_i-1}, &  \mbox{if} \ i=i_0.\\
\end{cases}$$
By the definition of $I_i, \ 1\le i\le 4$, one has
$$c_ic_{\pi(i)} = \begin{cases}c_i^2=1=h_{{\bf e}_i}, & \mbox{if} \ \ i\in I_1\cup I_2;\\
    c_i^2=-1=h_{{\bf e}_i}, & \mbox{if} \ \ i\in I_3\cup I_4-\{i_0\};\\
    1\cdot c_{\pi(i)}= h_{{\bf e}_{\pi(i)}}=h_{{\bf e}_i}, & \mbox{if} \ \ i\in J, \ i<\pi(i);\\
    h_{{\bf e}_i}c_{\pi(i)}=h_{{\bf e}_i}\cdot 1 = h_{{\bf e}_i} , & \mbox{if} \ \ i\in J, \ i>\pi(i).
\end{cases}$$
Thus $c_ic_{\pi(i)}=h_{{\bf e}_i}$ for $i\in \{1,\cdots, n\}-\{i_0\}$. By Lemma \ref{lemmaeven}(1) $|I_3|$ is even; and by definition $a_i$ is even for $i\in I_3$, thus $(|I_3|-1)(a_i-1)$ is odd for $i\in I_3$. Hence
$\prod\limits_{i\in I_3-\{i_0\}}c_i^{a_i-1}=\pm \sqrt{-1}$. Notice that $a_i-1$ is even for $i\in I_4$ and that $h_{{\bf e}_i}=\pm 1$ for $1\le i\le n$. Thus
\begin{align*}
    c_{i_0} &= (-1)^{\frac{a_{i_0}}{2}}q_{\pi}\prod\limits_{i\in \{1,\cdots, n\}-\{i_0\}}c_i^{a_i-1}\\
    & = \pm (-1)^{\frac{a_{i_0}}{2}}\prod\limits_{i\in I_1\cup I_2}c_i^{a_i-1} \prod\limits_{i\in I_3-\{i_0\}}c_i^{a_i-1}\prod\limits_{i\in I_4}c_i^{a_i-1}\prod\limits_{i\in J}c_i^{a_i-1} \\
    &=\pm \sqrt{-1}.
\end{align*}
It follows that $c_{i_0}c_{\pi(i_0)}=c_{i_0}^2=-1 = h_{{\bf e}_{i_0}}$, where the last equality follows from the definition of $I_3$ and $i_0\in I_3$.
\vskip5pt
Moreover, notice that $a_{i_0}$ is even, since $i_0\in I_3$. Thus
\begin{align*}q_\pi \prod\limits_{1\le i\le n}c_i^{a_i-1}& = q_\pi c_{i_0}^{a_{i_0}-1}\prod\limits_{i\in \{1,\cdots, n\}-\{i_0\}}c_i^{a_i-1}
\\    &= c_{i_0}^{a_{i_0}-1}q_\pi\prod\limits_{i\in \{1,\cdots, n\}-\{i_0\}}c_i^{a_i-1}
\\    &= c_{i_0}^{a_{i_0}-1}((-1)^{\frac{a_{i_0}}{2}}c_{i_0})
= (-1)^{\frac{a_{i_0}}{2}}c_{i_0}^{a_{i_0}}
 = (-1)^{\frac{a_{i_0}}{2}}(c_{i_0}^2)^{\frac{a_{i_0}}{2}}
\\ &= (-1)^{\frac{a_{i_0}}{2}}(-1)^{\frac{a_{i_0}}{2}} = (-1)^{a_{i_0}} = 1.
\end{align*}
By Theorem \ref{mainthm},  $A$ admits a bi-Frobenius algebra structure with permutation antipode.

\vskip5pt

{\bf Case 3:  \ $|I_1|+|I_3|= 0$}.  In this case, by Lemma \ref{lemmaeven}(2) one has $|I_4| = 0$ and $I = I_2$. By assumption $\frac{|J_3|}{2}$ is even.  We choose $c_i$ as follows:
$$c_i = \begin{cases}
    1, & \mbox{if} \ i\in I = I_2;\\
    1, & \mbox{if} \ i\in J, \ i<\pi(i);\\
    h_{{\bf e}_i}, & \mbox{if} \ i\in J, \ i>\pi(i).
\end{cases}$$
By the definition of $I_2$,   $h_{{\bf e}_i} = 1$ for $i\in I_2$. Thus
$$c_ic_{\pi(i)} = \begin{cases}c_i^2=1=h_{{\bf e}_i}, & \mbox{if} \ \ i\in I = I_2;\\
    1\cdot c_{\pi(i)}= h_{{\bf e}_{\pi(i)}}=h_{{\bf e}_i}, & \mbox{if} \ \ i\in J, \ i<\pi(i);\\
    h_{{\bf e}_i}c_{\pi(i)} = h_{{\bf e}_i}\cdot 1 = h_{{\bf e}_i} , & \mbox{if} \ \ i\in J, \ i>\pi(i),
\end{cases}$$
i.e., $c_ic_{\pi(i)}=h_{{\bf e}_i}$ for $1\le i\le n$.  Since $I = I_2$, it follows that  $$q_{\pi}=\prod\limits_{j, k \in I,\ j <k} (q_{kj})^{(a_k-1)(a_j-1)}=1=\prod\limits_{j, k \in I_2,\ j <k} (q_{kj})^{(a_k-1)(a_j-1)}=1$$
since $ q_{jk}^2 = 1$ for $j,k\in I$ (cf. Lemma \ref{lemmaI}(2)) and $a_i-1$ is even for $i\in I_2$. Hence
\begin{align*}q_\pi \prod\limits_{1\le i\le n}c_i^{a_i-1}&  =\prod\limits_{1\le i\le n} c_i^{a_i-1}
    = \prod\limits_{i\in I} c_i^{a_i-1}\prod\limits_{i\in J} c_i^{a_i-1}
    \\&=\prod\limits_{i\in J} c_i^{a_i-1} \xlongequal{{\rm Lem. \ref{lemmaI} (1)}} \prod\limits_{i\in J, i< \pi(i)} c_i^{a_i-1}\prod\limits_{i\in J, i< \pi(i)} c_{\pi(i)}^{a_{\pi(i)}-1}
    \\&= \prod\limits_{i\in J,\ i< \pi(i)} (c_ic_{\pi(i)})^{a_i-1} =\prod\limits_{i\in J,\ i< \pi(i)} h_{{\bf e}_i}^{a_i-1}
    \\& =\prod\limits_{i\in J_1\cup J_2\cup J_4, i< \pi(i)} h_{{\bf e}_i}^{a_i-1}\prod\limits_{i\in J_3, i< \pi(i)} h_{{\bf e}_i}^{a_i-1}
    \\ & = \prod\limits_{i\in J_3,\ i< \pi(i)} h_{{\bf e}_i}^{a_i-1} = \prod\limits_{i\in J_3,\ i<\pi(i)} (-1)^{a_i-1}
    \\
    & =(-1)^{ \frac{|J_3|}{2}}=1.
\end{align*}
By Theorem \ref{mainthm},  $A$ admits a bi-Frobenius algebra structure with permutation antipode.

\vskip5pt

Finally, the last sentence of Theorem \ref{mainthm2} follows from Theorem \ref{mainthm}. This completes the proof.
\end{proof}

\subsection{\bf Main result II: the case of $\sqrt{-1}\notin \mathbb{K}$}

\begin{theorem} \label{mainthm3} \  Assume that ${\rm Char}\mathbb{K}\ne 2$ and $\sqrt{-1}\notin \mathbb{K}$. Then $A = A({\bf q,a })$ admits a bi-Frobenius algebra structure with permutation antipode
if and only if $\mathcal{N}^2=\Id$ and there is a compatible permutation $\pi\in S_n$ with $\pi^2 = \Id$ such that
the following conditions are satisfied:

{\rm (1)}  \ $|I_3|+|I_4|= 0,$

{\rm (2)} \ $|I_1|\ne 0$ or $\frac{|J_3|}{2}$ is even,

\noindent
where  $I_1, I_3, I_4$ and $J_3$ are defined as in $(\ref{equationIJ})$.
\vskip5pt

    If this is the case, then  $(A, \ x_{\bf a-1}^*, \ x_{\bf a-1}, \ S)$ is a bi-Frobenius algebra,
    where  the comultiplication $\Delta$ is given by $(6.3)$ and $S$ is given by $(\ref{equationS})$.
\end{theorem}

\begin{proof} \ If $A({\bf q,a })$ admits a bi-Frobenius algebra structure with permutation antipode. Then $\mathcal{N}^2=\Id$ by Proposition \ref{thenecessity}(iv).
    By Lemma \ref{IJ}, there is a compatible permutation $\pi\in S_n$ with $\pi^2 = \Id$ such that $|I_1|+|I_3|\ne 0$ or $\frac{|J_3|}{2}$ is even.
\vskip5pt
    By Theorem \ref{mainthm}, there are $c_i\in \mathbb{K}, \ 1\le i\le n,$ such that $c_ic_{\pi(i)}=h_{{\bf e}_i}, \ 1\le i\le n$. If there is $i\in I_3\cup I_4$, then
$$c_i^2 = c_ic_{\pi(i)} =h_{{\bf e}_i}= -1$$ since by definition $h_{{\bf e}_i}=-1$ for $i\in I_3\cup I_4$. This contradicts the assumption $\sqrt{-1}\notin\mathbb{K}$.
Thus $|I_3|+|I_4|=0$, in particular $|I_3| = 0$. Hence  $|I_1|\ne 0$ or $\frac{|J_3|}{2}$ is even. Again by Theorem \ref{mainthm}, $(A, \ x_{\bf a-1}^*, \ x_{\bf a-1}, \ S)$ is a bi-Frobenius algebra,
    where  the comultiplication $\Delta$ is given by $(6.3)$ and $S$ is given by $(\ref{equationS})$.

    \vskip5pt

    Conversely, assume that $\mathcal{N}^2=\Id$ and there is a compatible permutation $\pi\in S_n$ with $\pi^2 = \Id$ such that $|I_3|+|I_4|= 0$ and that $|I_1|\ne 0$ or $\frac{|J_3|}{2}$ is even.
The proof below is similar to the one for Theorem \ref{mainthm2}. For completeness we include the proof. By Theorem \ref{mainthm} and Lemma \ref{lemmaI}(4), it suffices  to show that there are $c_i\in\mathbb{K}, \ 1\le i\le n, $ such that
        \[
            c_ic_{\pi(i)}=h_{{\bf e}_i}, \ 1\le i\le n; \ \ \ \ q_\pi\prod\limits_{1\le i\le n} c_i^{a_i-1}=1
        \]
where $q_\pi = \prod_{j,k\in I, j <k} ({\bf q}^{\langle  \pi({\bf e}_k)|\pi({\bf e}_j) \rangle})^{(a_k-1)(a_j-1)} = \pm 1$.    By Lemma \ref{propusefuleq2}(4), $h_{{\bf e}_i} h_{{\bf e}_{\pi(i)}} = 1, \ 1\le i\le n$. Since $h_{{\bf e}_i}^2=1$ for $1\le i\le n$,  $h_{{\bf e}_{\pi(i)}}=h_{{\bf e}_i} = \pm 1$, $1\le i\le n$.

    \vskip5pt

    {\bf Case 1}:  $|I_3|+|I_4|= 0$ and\ $|I_1|\ne 0$. In this case $I=I_1\cup I_2$ and there is $i_0\in I_1$. Choose $c_i$ as follows:
    $$c_i = \begin{cases}1, & \mbox{if} \ \ i\in I_1\cup I_2-\{i_0\};\\
        1, & \mbox{if} \ \ i\in J, \ i<\pi(i);\\
        h_{{\bf e}_i}, & \mbox{if} \ \ i\in J, \ i>\pi(i);\\
        q_\pi\prod\limits_{i\in \{1,\cdots, n\}-\{i_0\}}c_i^{a_i-1}, &  \mbox{if} \ \ i=i_0.
    \end{cases}$$
     By the definition of $I_1$ and $I_2$, one has
    $$c_ic_{\pi(i)} = \begin{cases}c_i^2=1=h_{{\bf e}_i}, & \mbox{if} \ \ i\in I_1\cup I_2-\{i_0\};\\
        1\cdot c_{\pi(i)}\xlongequal{} h_{{\bf e}_{\pi(i)}}=h_{{\bf e}_i}, & \mbox{if} \ \ i\in J, \ i<\pi(i);\\
        h_{{\bf e}_i}c_{\pi(i)}=h_{{\bf e}_i}\cdot 1 = h_{{\bf e}_i}, & \mbox{if} \ \ i\in J, \ i>\pi(i).
    \end{cases}$$
    This shows that $c_ic_{\pi(i)}=h_{{\bf e}_i}$ for $i\in \{1,\cdots, n\}-\{i_0\}$. Notice that $h_{{\bf e_i}}=\pm 1$ for $1\le i\le n$.
    Thus
    \begin{align*}
        c_{i_0}&=q_\pi\prod\limits_{i\in \{1,\cdots, n\}-\{i_0\}}c_i^{a_i-1}\\
        &  \xlongequal{{\rm Lem.} \ref{lemmaI}(4)}  \pm \prod\limits_{i\in I_1\cup I_2 -\{i_0\}}c_i^{a_i-1}\prod\limits_{i\in J}c_i^{a_i-1}=  \pm 1.
    \end{align*}
    Therefore $c_{i_0}c_{\pi(i_0)}=c_{i_0}^2= 1 = h_{{\bf e}_{i_0}}$, where the last equality follows from the definition of $I_1$ and $i_0\in I_1$.

    \vskip5pt
    Moreover, note that $a_{i_0}$ is even. Thus
    \begin{align*}q_\pi \prod\limits_{1\le i\le n}c_i^{a_i-1}& = q_\pi c_{i_0}^{a_{i_0}-1}\prod\limits_{i\in \{1,\cdots, n\}-\{i_0\}}c_i^{a_i-1} \\  & =  c_{i_0}^{a_{i_0}-1} (q_\pi\prod\limits_{i\in \{1,\cdots, n\}-\{i_0\}}c_i^{a_i-1}) \\  & = c_{i_0}^{a_{i_0}-1} c_{i_0} = c_{i_0}^{a_{i_0}} = (\pm 1)^{a_{i_0}} =1.
    \end{align*}
   Thus, by Theorem \ref{mainthm},  $A$ admits a bi-Frobenius algebra structure with permutation antipode.

    \vskip5pt

    {\bf Case 2}:  $|I_3|+|I_4|= 0$ and\ $|I_1|= 0$.  In this case  $I = I_2$. By assumption one has $\frac{|J_3|}{2}$ is even.  We choose $c_i$ as follows:
    $$c_i = \begin{cases}
        1, & \mbox{if} \ i\in I = I_2;\\
        1, & \mbox{if} \ i\in J, \ i<\pi(i);\\
        h_{{\bf e}_i}, & \mbox{if} \ i\in J, \ i>\pi(i).
    \end{cases}$$
    By the definition of $I_2$,   $h_{{\bf e}_i} = 1$ for $i\in I_2$. Thus
    $$c_ic_{\pi(i)} = \begin{cases}c_i^2=1=h_{{\bf e}_i}, & \mbox{if} \ \ i\in I = I_2;\\
        1\cdot c_{\pi(i)}= h_{{\bf e}_{\pi(i)}}=h_{{\bf e}_i}, & \mbox{if} \ \ i\in J, \ i<\pi(i);\\
        h_{{\bf e}_i}c_{\pi(i)}h_{{\bf e}_i}\cdot 1 = h_{{\bf e}_i} , & \mbox{if} \ \ i\in J, \ i>\pi(i),
    \end{cases}$$
    i.e., $c_ic_{\pi(i)}=h_{{\bf e}_i}$ for $1\le i\le n$.

    \vskip5pt

    Since $I = I_2$, it follows that  $$q_{\pi}=\prod\limits_{j, k \in I,\ j <k} (q_{kj})^{(a_k-1)(a_j-1)}=1=\prod\limits_{j, k \in I_2,\ j <k} (q_{kj})^{(a_k-1)(a_j-1)}=1$$
    since $ q_{jk}^2 = 1$ for $j,k\in I$ (cf. Lemma \ref{lemmaI}(2)) and $a_i-1$ is even for $i\in I_2$. Hence
    \begin{align*}q_\pi\prod\limits_{1\le i\le n} c_i^{a_i-1}
        &=\prod\limits_{1\le i\le n} c_i^{a_i-1}
        = \prod\limits_{i\in I} c_i^{a_i-1}\prod\limits_{i\in J} c_i^{a_i-1}
        \\&=\prod\limits_{i\in J} c_i^{a_i-1} \xlongequal{{\rm Lem. \ref{lemmaI} (1)}} \prod\limits_{i\in J, i< \pi(i)} c_i^{a_i-1}\prod\limits_{i\in J, i< \pi(i)} c_{\pi(i)}^{a_{\pi(i)}-1}
        \\&= \prod\limits_{i\in J,\ i< \pi(i)} (c_ic_{\pi(i)})^{a_i-1} =\prod\limits_{i\in J,\ i< \pi(i)} h_{{\bf e}_i}^{a_i-1}
        \\& =\prod\limits_{i\in J_1\cup J_2\cup J_4, i< \pi(i)} h_{{\bf e}_i}^{a_i-1}\prod\limits_{i\in J_3, i< \pi(i)} h_{{\bf e}_i}^{a_i-1}
        \\ & = \prod\limits_{i\in J_3,\ i< \pi(i)} h_{{\bf e}_i}^{a_i-1} = \prod\limits_{i\in J_3,\ i<\pi(i)} (-1)^{a_i-1}
        \\
        & =(-1)^{ \frac{|J_3|}{2}}=1.
    \end{align*}
    By Theorem \ref{mainthm},  $A$ admits a bi-Frobenius algebra structure with permutation antipode. This completes the proof.
\end{proof}

\vskip20pt

{\bf Acknowledgement}: The authors thank the anonymous referee for helpful comments and suggestions towards the presentation of the paper.

\vskip5pt

\centerline {\bf Declarations}

\vskip5pt

Conflict of interest. The authors state that there are no competing interests to declare.

\bibliographystyle{alpha}

\begin{thebibliography}{99}
\bibitem[Mar]{Mar} R. Marczinzik, On stable modules that are not Gorenstein projective, ArXiv:1709.01132v3.
\end{thebibliography}

\end{document}